\documentclass[11pt,reqno]{amsart}
\RequirePackage[table]	{xcolor}			
\definecolor{green1}{RGB}{0,107,28}
\definecolor{green2}{RGB}{17,85,35}
\definecolor{green3}{RGB}{0,77,20}
\definecolor{green4}{RGB}{33,166,68}
\definecolor{green5}{RGB}{60,166,88}

\definecolor{blue1}{RGB}{3,56,91}
\definecolor{blue2}{RGB}{16,51,73}
\definecolor{blue3}{RGB}{1,40,66}
\definecolor{blue4}{RGB}{35,109,157}
\definecolor{blue5}{RGB}{59,118,157}
\RequirePackage[
	final,							
	colorlinks=true,
	urlcolor=blue3,						
	linkcolor=blue3,					
	citecolor=green4,					
	hypertexnames=false,					
	hyperfootnotes=false,					
	]{hyperref}						

\usepackage{a4wide}
\usepackage{latexsym}
\usepackage{amsopn,amscd}
\usepackage{amsfonts}
\usepackage{amssymb}
\usepackage{amsthm}
\usepackage{amsmath}
\usepackage[all]{xy}
\usepackage{scalefnt}

\usepackage[capitalise,nameinlink,noabbrev]{cleveref}
\usepackage{hyperref}

\def\Hom{\mathrm{Hom}}

\def\JJJ{\mathrm J}

\def\pd{\partial}

\def\dell{\mathcal D}
\def\ppi{g}

\def\chr{\mathrm{Chain}_R}
\def\crg{\mathrm{Chain}_{R \mathcal G}}
\def\modr{\mathrm{Mod}_R}
\def\Cat{\mathcal C}
\def\XX{X}
\def\DDD{\mathcal D}

\def\GGG{\mathcal G}
\def\UUU{\mathcal U}
\def\TTT{\mathcal T}
\def\PPP{\mathcal P}

\def\GG{G}
\def\Map{\mathrm{Map}}

\def\MMM{\mathcal M}
\def\MU{\MMM_\UUU}

\def\mor{\mathrm{Mor}}
\def\ob{\mathrm{Ob}}
\def\Id{\mathrm{Id}}

\def\NNN{N}

\def\Ord{\mathrm{Ord}}
\def\Mor{\mathrm{Mor}}
\def\Mort{\Mor_{\mathrm{tw}}}
\def\mult{\mu}

\def\KK{K}
\def\last{\mathrm{last}}

\def\Moo{\mathrm{Mo}}
\def\TTT{\mathcal T}
\def\FF{\mathrm{F}}

\newtheorem{thm}{Theorem}[section]
\newtheorem{cor}[thm]{Corollary}
\newtheorem{lem}[thm]{Lemma}
\newtheorem{prop}[thm]{Proposition}

\newtheorem{examp}[thm]{Example}
\newtheorem{rema}[thm]{Remark}

\newtheorem{compl}[thm]{Complement}

\long
\def\MSC#1\EndMSC{\def\arg{#1}\ifx\arg\empty\relax\else
      {\par\narrower\noindent
      2020 Mathematics Subject Classification. #1\par}\fi}

\long
\def\KEY#1\EndKEY{\def\arg{#1}\ifx\arg\empty\relax\else
    {\par\narrower\noindent
      Keywords and Phrases: #1\par}\fi}

\newcommand{\n}[1]{\nobreakdash-\hspace{0pt}}

\newcommand{\ainf}[1]{$A_\infty$\nobreakdash-\hspace{0pt}}

\DeclareMathOperator\Ob{Ob}
\DeclareMathOperator\pr{pr}

\numberwithin{equation}{section}

\long
\def\MSC#1\EndMSC{\def\arg{#1}\ifx\arg\empty\relax\else
      {\par\narrower\noindent
      2020 Mathematics Subject Classification. #1\par}\fi}

\long
\def\KEY#1\EndKEY{\def\arg{#1}\ifx\arg\empty\relax\else
    {\par\narrower\noindent
      Keywords and Phrases: #1\par}\fi}

\title 
 {The breadth of  Berikashvili's functor $\DDD$}

\author{Johannes Huebschmann  }
\address{
\newline
Universit\'e de Lille - Sciences et Technologies 
\\
D\'epartement de Math\'ematiques\\
\newline CNRS-UMR 8524,
Labex CEMPI (ANR-11-LABX-0007-01)
\\
\newline
59655 Villeneuve d'Ascq Cedex, France\\
\newline
Johannes.Huebschmann@univ-lille.fr
 }

\date{\today}
\numberwithin{equation}{section}

\begin{document}
\setcounter{page}{1}

\begin{abstract}
\noindent
We discuss variants of Berikashvili's functor 
that arise in differential homological algebra,
from simplicial bundles,  from ordinary topological bundles,
and in more general categorical settings.
We prove that, under suitable circumstances,
 the value of Berikasvili's functor
parametrizes isomorphism classes of bundles in various contexts.
\end{abstract}
\maketitle

\centerline{\em To the memory of Nodar Berikashvili} 
\bigskip

\MSC

\noindent
Primary: 16E45 

\noindent
Secondary:  18G31, 18G35,  18G50, 18N50, 55R10,  55R20, 55U10
\EndMSC

\KEY
Twisting cochain, twisted tensor products and their classification, 
simplicial principal bundles and their classification,
moduli space of flat connections, first nonabelian cohomology
 \EndKEY
{\tableofcontents}

\section{Introduction}

In \cite{MR0264667},
N. Berikashvili 
introduced the 
functor $\DDD$ 
in terms of \lq\lq twisting elements\rq\rq\  or
\lq\lq twisting cochains\rq\rq\  in a differential graded algebra.
At that time 
twisting cochains already had a history in topology and differential
homological algebra,
see e.~g.
\cite {MR0436178} (where the terminology is \lq\lq twisting morphism\rq\rq)
and the  references there to \cite{MR0105687}
and \cite{MR0042426, MR0042427}.
With hindsight we see that twisting cochains make precise a certain
piece of structure behind the notion of transgression
\cite{MR0042426, MR0042427}.
The original version of Berikashvili's functor $\DDD$ assigns
to a differential graded algebra $\mathcal A$
the set $\DDD(\mathcal A)=\TTT (\mathcal A)/ \GGG $ of orbits
of twisting elements (homogeneous degree $-1$ members $\tau$ of $\mathcal A$ 
that satisfy the identity $d\tau = \tau \tau$ or master equation
relative to the differential $d$ in $\mathcal A$) in 
$\mathcal A$ with respect to the group $\GGG $ of inner automorphisms
of $\mathcal A$. 
For a differential graded algebra $\mathcal A$,
the set $\DDD(\mathcal A)$ 
ressembles  a {\em moduli space of gauge 
equivalence classes of 
flat connections\/}.
Such a moduli 
space is sometimes considered as a first non-abelian coomology space.
The results in the present paper suggest that, in 
the situations we discuss here,
 the value of Berikashvili's functor
could also be viewed as a first non-abelian cohomology set.
Remark \ref{nonabcoh} renders this observation explicit.

The paper \cite{MR1710565}
explains some of  the significance of Berikashvili's functor
within ordinary differential homological algebra
and how it relates to  deformation theory,
and \cite{MR3956729} develops a small aspect of that relationship
further.
Here we discuss variants of Berikashvili's functor in various contexts.

We begin by reviewing
bundles in the category of chain complexes.
Applying Berikashvili's functor
$\DDD$ to 
the differential graded algebra $\Hom(C,A)$ associated to
a differential graded coalgebra $C$ and a differential 
graded algebra $A$
(endowed with the cup or convolution product)
yields  the set $\DDD(C,A)$
of orbits of twisting cochains $\tau \colon C \to A$
relative to the group of inner automorphisms of
the differential graded algebra $\Hom(C,A)$.
We  prove
(Theorem \ref{param1}) 
that, for
a differential graded coalgebra $C$ and a differential 
graded algebra $A$, the set
 $\DDD(C,A)$
parametrizes isomorphism classes of bundles (principal
 twisted tensor products) having $C$ as base and
 $A$ as fiber.
Thereafter,
guided by the idea that, 
for a group $\GG$,
a principal $\GG$-bundle 
admits a characterization
in terms of a certain $\GG$-valued functor
defined on a certain category, cf. \cite{MR0232393},
 we proceed in a more abstract  manner
within the appropriate categorical framework. 
Theorem \ref{thm4} 
says that, 
for a simplicial set $B$ and a simplicial group $\KK$,
the value  $\DDD(B,\KK)$ of Berikash\-vili's functor on the pair $(B,\KK)$
parametrizes isomorphism classes of
simplicial principal $\KK$-bundles 
on $B$.
 Theorem \ref{thm5} establishes the fact
that, for a simplicial set $B$ and a simplicial group $\KK$,
 the assignment
to a twisting function $\rho\colon B \to \KK$ of the  twisting cochain
which $\rho$ determines via the perturbation lemma
yields a map from isomorphism classes of simplicial principal bundles
to the isomorphism classes of the associated twisted tensor products.
Theorem \ref{bericat}
yields the same kind of parametrization result, but  phrased over a category.
This recovers
principal bundles over a simplicial complex
(Example \ref{examp1}) and 
over the simplicial category arising from
an open cover of a manifold  
(Example \ref{examp2}) 
and hence
ordinary principal bundles
over a manifold endowed with a partition of unity subordinate to the
open cover:
A construction in \cite{MR0232393}
assigns to an  open cover $\UUU$ of a manifold $M$
a (topological or smooth) category $\MU$ in such a way that a functor
$F$ from  $\MU$ to a group $\GG$
(viewed as a category in the standard way)
determines a simplicial principal $\GG$-bundle
on the nerve of $\MU$. 
Theorem \ref{bericatG} says that evaluating Berikasvili's functor
$\DDD$ at $\MU$ and $\GG$ yields a set
$\DDD(\MU,\GG)$ which recovers the non-abelian cohomology set
$\mathrm H^1(B \UUU,\GG)$
of equivalence classes of $\GG$-transition functions
relative to the open cover $B \UUU$ of $M$
arising as the  \lq barycentric subdivision\rq\ 
$B \UUU$ of $\UUU$; when 
$M$ admits a partition of unity subordinate to
$B \UUU$, in particular, when $M$ is paracompact,
 by a classical result, this cohomology set
characterizes isomorphism classes of principal $\GG$-bundles
on the  manifold $M$.
Section \ref{tcs} explores Berikashvili's functor in a general 
categorical setting.
This relates Berikashvili's functor to
 $A_{\infty}$-functors and provides, perhaps,
clarification and simplification 
for the theory of dg categories. For example,
the paper \cite{MR2028075} explores dg quotients of dg categories; in
 \cite[p.~687, p.~689]{MR2028075}, the terminology is
\lq Maurer-Cartan functor\rq\  for the
assignment to a coalgebra and algebra of
the family of twisting cochains between the two.
Thus there is a huge
unexplored territory in this area.

\section{Preliminaries}

References for  notation and terminology are
\cite{MR0345115, MR0402769, MR0056295,  MR0065162,  MR112135, 
MR0301736, MR662761, 
MR394720, MR0347946,MR1710565,
MR2640649, MR2762544,
MR2762538,
MR2820385,
MR0365571, 
MR1109665,  MR1932522, MR0436178, MR0413090,  MR350735, MR0258031}.

The ground ring $R$ is a commutative ring with $1$.
We  take {\em chain complex\/} to mean {\em
differential graded\/} $R$-{\em module\/}. A chain 
complex is
not necessarily  concentrated in non-negative or non-positive
degrees. The differential of a chain complex is always 
 of degree $-1$. For a filtered chain complex $X$, a
{\em perturbation\/} of the differential $d$ of $X$ is a
(homogeneous) morphism $\partial$ of the same degree as $d$ such
that $\partial$ lowers  filtration and $(d +
\partial)^2 = 0$ or, equivalently,
\begin{equation}
[d,\partial] + \partial \partial = 0.
\end{equation}
Thus, when $\partial$ is a perturbation on $X$, the sum $d +
\partial$, referred to as the {\em perturbed differential\/},
endows $X$ with a new differential. When $X$ has a graded
coalgebra structure such that $(X,d)$ is a differential graded
coalgebra, and when the {\em perturbed differential\/} $d +
\partial$ is compatible with the graded coalgebra structure, we
refer to $\partial$ as a {\em coalgebra perturbation\/}; the
notion of {\em algebra perturbation\/} is defined similarly. Given
a differential graded coalgebra $C$ and a coalgebra perturbation
$\partial$ of the differential $d$ on $C$,  occasionally we
denote the new or {\em perturbed\/} differential graded coalgebra
by $C_{\partial}$.
Given a differential graded algebra $A$ and an algebra perturbation
$\partial$ of the differential $d$ on $A$,  occasionally we
likewise
denote the new or {\em perturbed\/} differential graded algebra
by $A_{\partial}$.

Given two chain complexes $X$ and $Y$, recall that
$\mathrm{Hom}(X,Y)$ inherits the structure of a chain complex by
the operator $D$ defined by
\begin{equation}
D \phi = d \phi -(-1)^{|\phi|} \phi d
\end{equation}
where $\phi$ is a homogeneous homomorphism from $X$ to $Y$ and
where $|\phi|$ refers to the degree of $\phi$.
The notation $D$ for the Hom-differential
and $\DDD$ for Berikashvili's functor
might be slightly confusing but is, perhaps,
unavoidable, for reasons of consistency with notation established 
in the literature.

A {\em contraction\/}
\begin{equation}
\xymatrix{(M
\ar@<0.5ex>[r]^{\nabla\phantom{}}
&\ar@<0.5ex>[l]^{\ppi\phantom{}}
N,h)
} 
\label{co}
\end{equation}
of chain complexes \cite{MR0056295} consists of

-- chain complexes $N$ and $M$,
\newline
\indent -- chain maps $\ppi\colon N \to M$ and $\nabla \colon M \to
N$,
\newline
\indent --  a morphism $h\colon N \to N$ of the underlying graded
modules of degree 1,
\newline
\noindent subject to
\begin{align}
 \ppi \nabla &= \mathrm{Id},
\label{co0}
\\
Dh &= \nabla \ppi - N, \label{co1}
\\
\ppi h &= 0, \quad h \nabla = 0,\quad hh = 0. \label{side}
\end{align}
It is common to refer to 
the requirements \eqref{side}  as {\em annihilation
properties\/} or {\em side conditions\/}. We say a 
contraction \eqref{co} having $M$ and $N$ filtered chain complexes 
and $\nabla$, $g$, $h$ filtration preserving is {\em filtered\/}.

\begin{rema}
{\rm
Given $M$, $N$, $\ppi$, $\nabla$, $h$, subject to \eqref{co0} and
\eqref{co1} but not necessarily \eqref{side}, substituting
$\ppi - \ppi d h$ and $hdh$ 
for $\ppi$ and $h$, we obtain
a contraction in the strict sense, that is, the
data satisfy \eqref{side} as well.
}
\end{rema}

For later reference, we recall the 
ordinary perturbation lemma.

\begin{lem} 
\label{ordinary} 
Let 
\begin{equation}
\xymatrix{(M
\ar@<0.5ex>[r]^{\nabla\phantom{}}
&\ar@<0.5ex>[l]^{\ppi\phantom{}}
N,h)
} 
\label{266}
\end{equation}
be a filtered contraction. 
Let $\partial$ be a perturbation of the differential
on $N$, and let
\begin{align} \dell &= \sum_{n\geq 0} \ppi\partial (h\partial)^n\nabla =
\sum_{n\geq 0} \ppi(\partial h)^n\partial\nabla
\label{1.1.1}
\\
\nabla_{\partial}&= \sum_{n\geq 0} (h\partial)^n\nabla
\label{1.1.2}
\\
\ppi_{\partial}&= \sum_{n\geq 0} \ppi(\partial h)^n
\label{1.1.3}
\\
h_{\partial}&=\sum_{n\geq 0} (h\partial)^n h =\sum_{n\geq 0}
h(\partial h)^n .
\label{1.1.4}
\end{align}
When the filtrations on $M$ and $N$ are complete, these infinite
series converge, the operator $\dell$ is a perturbation of the differential on
$M$ and, with the notation $N_{\partial}$ and $M_{\dell}$ for  the new
chain complexes,
\begin{equation}
\xymatrix{(M_\dell
\ar@<0.5ex>[r]^{\nabla_\partial\phantom{}}
&\ar@<0.5ex>[l]^{\ppi_\partial\phantom{}}
N_\partial,h_\partial)
} 
\label{2.66}
\end{equation}
constitute a new filtered contraction that is natural in terms of
the given data.
\end{lem}

\begin{proof} See \cite{MR0220273} or
\cite[4.3 Lemma  Section 4 p.~404]{MR0301736}.
\end{proof}

Under the circumstances of Lemma \ref{ordinary}, we refer to
\eqref{2.66} as the {\em perturbed contraction\/}.

\begin{rema}
{\rm 
In \eqref{co1}, the
sign of $h$ is the same
as in \cite[ Section 4 p.~403]{MR0301736}
and \cite[Section 1 p.~247]{MR1109665},
opposite to that in \cite[2.2 p.~164]{MR2640649}
(in (2.2) of that paper,  $M$ and $N$ should be  exchanged).
This explains the appearance of the minus sign in
 \cite[(9.2) - (9.5), p.~183]{MR2640649}.
}
\end{rema}
Let $(C,\eta,\Delta)$ be a 
{\em coaugmented\/} differential graded coalgebra
and let 
$\JJJ C = \mathrm{coker}(\eta)$ denote its {\em
coaugmentation\/} coideal. 
 Recall  the counit $\varepsilon \colon C \to R$ and the
coaugmentation map $\eta$ determine a direct sum decomposition $C = R
\oplus \JJJ C$, and the diagonal $\Delta$ 
induces a diagonal $\JJJ \Delta \colon \JJJ C \to \JJJ C \otimes \JJJ C$.
The 
ascending sequence 
\begin{equation}
R \subseteq \ldots \subseteq  \mathrm F_nC\subseteq
\mathrm F_{n+1}C \subseteq \ldots \ (n \geq 1)
\end{equation}
formed by the kernels
\[\mathrm F_nC = \mathrm{ker}(C\stackrel{\pr}\longrightarrow \JJJ C 
\stackrel{(\JJJ \Delta)^{\otimes (n+1)}}\longrightarrow 
(\JJJ C)^{\otimes (n+1)})\  (n \geq 0)
\]
yields the {\em coaugmentation\/} filtration $\{\mathrm
F_nC\}_{n \geq 0}$ of $C$,
cf., e.g.,
\cite[Section 1 p.~11]{MR0347946},
well known to turn
$C$ into a {\em filtered\/} coaugmented differential graded
coalgebra; thus, in particular, $\mathrm F_0C = R$. We recall that
$C$ is said to be {\em cocomplete\/} when $C=\cup \mathrm F_nC$.

Write $s$ for the {\em suspension\/} operator and
accordingly $s^{-1}$ for the {\em desuspension\/} operator. Thus,
given the chain complex $X$, $(sX)_j = X_{j-1}$, etc., and the
differential ${d\colon sX \to sX}$ on the suspended object $sX$ is
defined in the standard manner so that $ds+sd=0$.

Consider  a simplicial $R$-module $W$.
We take its {\em Moore complex\/}
$(\Moo(W),d)$
to be
\begin{equation}
\begin{aligned}
\Moo(W)&\colon \ldots \stackrel{d_{p+1}}\longrightarrow \Moo_p(W)\stackrel{d_p} \longrightarrow \ldots \stackrel{d_1}\longrightarrow \Moo_0(W) = W_0,\ 
\\
\Moo_p(W) &= \cap_{0 \leq j <p} \ker (\pd_j\colon W_p \to W_{p-1}),\ p \geq 1, 
\\
d_p&= \pd_p|_{ \Moo_p(W)},\ p \geq 1.
\end{aligned}
\end{equation}
The canonical injection 
$\iota \colon \Moo(W) \to W$ extends to a contraction
\begin{equation}
 (\negthickspace\xymatrix{\Moo(W)
\ar@<0.5ex>[r]^{\phantom{a}\iota}
&\ar@<0.5ex>[l]^{\phantom{a}\ppi}
{W} ,
h 
}\negthickspace) 
\label{contramoore}
\end{equation}
that is natural in terms of the data,
the kernel of $\ppi$ coincides with the degeneracy subcomplex 
$D(W)$
of $W$, and $\ppi$ induces an isomorphism
$|W| =W/D(W) \to \Moo(W)$
from the normalized $R$-chain complex
$|W|$ onto the Moore complex $\Moo(W)$ of $W$;
see, e.g.,
\cite[Ch. III Theorem 2.1  p.~146]{MR1711612}.
We identify 
$|W|$ with $\Moo(W)$.

\section{Bundles in the category of chain complexes}
\label{bundles}
The classical notion of {\em twisting cochain\/} goes back to \cite{MR0105687};
this notion arises by abstraction from properties of the transgression
\cite{MR0042427, MR0036511, MR42428}. 
The defining equation of a twisting cochain
also occurs in the literature as
{\em master equation\/}; see \cite{MR1932522} and 
the literature there.

Let $C$ be a differential graded coalgebra and $A$
a differential graded algebra.
Write the unit as $\eta \colon R \to A$, the counit as
$\varepsilon \colon C \to R$, the multiplication map
of $A$ as $\mult \colon A \otimes A \to A$,
and the diagonal map
of $C$ as $\Delta \colon C \to C \otimes C$.
The  Hom differential $D$ and {\em cup product\/}
\begin{equation*}
\begin{aligned}
\cup&\colon \Hom(C,A) \otimes \Hom(C,A) \longrightarrow \Hom(C,A),
\\
\alpha \cup \beta&\colon C \stackrel \Delta \longrightarrow C \otimes C
 \stackrel{\alpha \otimes \beta}  \longrightarrow A\otimes A
\stackrel{\mult}  \longrightarrow A, \ \alpha, \beta \colon C \to A,
\end{aligned}
\end{equation*}
\cite[II.1.1 Definition p.~135]{MR0365571}
turn  $\Hom(C,A)$ into a differential graded algebra
having unit the composite $\eta  \varepsilon $.
The cup product has also come to be known as
{\em convolution product\/}
\cite[p.~6]{MR1243637}.
When $C$ is  coaugmented,
 with coaugmentation
$\eta \colon R \to C$,  and $A$  augmented 
with augmentation
$\varepsilon \colon A \to R$,
\begin{equation}
\Hom(C,A) \longrightarrow R,\ \varphi \mapsto 
\varepsilon \varphi  \eta
\end{equation}
is an augmentation map for $\Hom(C,A)$ as a differential graded algebra.

A $(C,A)$-{\em bundle\/}
is a differential graded right $A$-module left $C$-comodule
$N$ together with
an isomorphism $\lambda \colon N \to C \otimes A$
of  graded right $A$-modules and left $C$-comodules, 
differentials being neglected, the right $A$-module and left $C$-comodule
structures on $C \otimes A$ being the extended ones
\cite[II.1.6 p.~137]{MR0365571}.
For a morphism $\phi\colon A_1 \to A_2$ of 
differential graded algebras
and a morphism $\psi \colon C_1 \to C_2$ of 
differential graded coalgebras,
a $(\phi,\psi)$-{\em morphism\/} 
\begin{equation}
(C_1 \otimes A_1, d_1) \longrightarrow (C_2 \otimes A_2, d_2)
\end{equation}
from the $(C_1,A_1)$-bundle 
$(C_1 \otimes A_1, d_1)$
to the $(C_2,A_2)$-bundle $(C_2 \otimes A_2, d_2)$
is a morphism of chain complexes which is, furthermore,
a morphism of differential graded  $A_1$-modules via $\phi\colon A_1 \to A_2$
and of differential graded
$C_2$-comodules via $\psi\colon C_1 \to C_2$.
With the obvious notions of composition of morphisms and identity,
 $(C,A)$-bundles and, more generally, bundles constitute a category.
In particular, isomorphism classes of bundles
are well defined.

For an augmented differential graded algebra $(A,\varepsilon)$,
we say a $(C,A)$-bundle $(C\otimes A, d)$ 
is {\em augmented\/} when
$C \otimes \varepsilon \colon (C \otimes A,d) \to C$
is a morphism of differential graded $C$-comodules.
For a coaugmented differential graded coalgebra $(C,\eta)$,
we say a $(C,A)$-bundle $(C\otimes A, d)$ 
is {\em coaugmented\/} when
$\eta \otimes A \colon A \to (C \otimes A,d)$
is a morphism of differential graded $A$-modules,
and we say
a $(C,A)$-bundle $(C\otimes A, d)$ 
is {\em supplemented\/} 
when it is both augmented and coaugmented.
Now, for a $(\phi,\psi)$-morphism 
$(C_1 \otimes A_1, d_1) \to
(C_2 \otimes A_2, d_2)$
of coaugmented bundles,
the diagram
\begin{equation}
\begin{gathered}
\xymatrix{
A_1\ar[r]^\phi\ar[d]_{\eta \otimes \Id} & A_2 \ar[d]^{\eta \otimes\Id }
\\
(C_1 \otimes A_1, d_1) \ar[r] & 
(C_2 \otimes A_2, d_2)
}\end{gathered}
\end{equation}
is commutative and,
for a $(\phi,\psi)$-morphism 
$(C_1 \otimes A_1, d_1) \to 
(C_2 \otimes A_2, d_2)$
of augmented bundles, the diagram
\begin{equation}
\begin{gathered}
\xymatrix{
(C_1 \otimes A_1, d_1) \ar[r] \ar[d]_{\Id \otimes \varepsilon}& 
(C_2 \otimes A_2, d_2) \ar[d]^{\Id \otimes \varepsilon}
\\
C_1\ar[r]_\psi & C_2
}
\end{gathered}
\end{equation}
is commutative.
In practice, morphisms of
bundles 
arising from simplicial fiber bundles
are  augmented; they are, furthermore, 
coaugmented after a choice of base point of the 
base.

We define
a homogeneous degree $-1$ 
morphism $\tau \colon C \to A$ 
of the underlying graded $R$-modules
to be a
{\em twisting cochain\/}
when 
\begin{equation}
D\tau +\tau \cup \tau=0.
\end{equation}
For a coaugmented differential graded coalgebra
$(C,\eta)$, we say
a twisting cochain $\tau \colon C \to A$ is {\em coaugmented\/}
when $\tau  \eta = 0$ and,
for an augmented differential graded algebra
 $(A,\varepsilon)$,
 we say
a twisting cochain $\tau \colon C \to A$ is {\em augmented\/}
when $ \varepsilon  \tau = 0$.
We refer to a twisting cochain
that is both coaugmented and augmented
as being {\em supplemented\/}.

\begin{rema}
{\rm This definition of a twisting cochain is that in
\cite[Section 2, p.~401]{MR0301736}. 
The terminology in \cite{MR0365571} is \lq twisting morphism\rq\ 
for a supplemented  twisting cochain,
 with  opposite sign, and that in
\cite{zbMATH03921525}
is twisting cochain,
still with  the sign opposite to the present one.
The present sign for a twisting cochain
is most convenient for bundles of the kind $C \otimes A$
(as opposed to those of the kind $A \otimes C$). 
}
\end{rema}

Let $d^\otimes$ denote the tensor product differential.
The cap product
\begin{equation}
\begin{aligned}
\cap &\colon \Hom(C,A) \otimes C \otimes A  \longrightarrow C \otimes A,
\\
\varphi \cap \,
\cdot\,  &\colon 
C \otimes A  \stackrel{\Delta \otimes A}\longrightarrow C \otimes C  \otimes A
\stackrel{C \otimes \varphi \otimes A}\longrightarrow C \otimes A  \otimes A
\stackrel{C \otimes \mult}\longrightarrow C \otimes A, 
\ \varphi \colon C \to A,
\end{aligned}
\label{action1}
\end{equation}
\cite[2.3 Definitions p.~401]{MR0301736}
yields an action of $\Hom(C,A)$
on $C \otimes A$.
For a twisting cochain $\tau \colon C \to A$, let
$d^\tau =  \cap \tau$. Then the sum
\begin{equation}
d^\otimes + d^\tau \colon C \otimes A \to C \otimes A
\end{equation}
is a differential on $C \otimes A$
which, relative to the obvious structures,
 turns $C \otimes A$ into a 
$(C,A)$-bundle.
It is common to denote this $(C,A)$-bundle by 
$C \otimes_\tau A$; we refer to it as an 
$A$-{\em principal twisted tensor product on\/}
$C$.
See, e.g.,
\cite{MR0365571} for details.
By construction, the $A$-principal twisted tensor product
$C \otimes_\tau A$ on $C$ is a  $(C,A)$-bundle
in a canonical manner,
the right $A$-module and left $C$-comodule
structures on $C \otimes A$ being the extended ones.
The terminology in
\cite{MR0301736}
is \lq\lq principal twisted object\rq\rq.
The following reproduces
\cite[2.2 Proposition p.~400]{MR0301736}.

\begin{prop}
\label{gug1}
For a differential graded algebra $A$
and a  differential graded coalgebra $C$,
the assignment to a bundle differential 
$D$ on $C \otimes A$ of the degree $-1$ morphism 
\begin{equation}
\tau_D\colon C 
\stackrel{\Id \otimes \eta} \longrightarrow C \otimes A
\stackrel D\longrightarrow C \otimes A
\stackrel{\varepsilon \otimes \Id} \longrightarrow  A 
\end{equation}
of the underlying graded objects 
establishes, in the general,
 augmented, coaugmented, and supplemented case,
 a bijection between
twisting cochains from $C$ to $A$ and
bundle differentials 
$D$ on $C \otimes A$ in such a way that
$(C\otimes A,D) = C\otimes _{\tau_D}A$.
Thus
any $(C,A)$-bundle structure on $C \otimes A$
is of the kind $C \otimes_\tau A$, for some uniquely determined twisting cochain 
$\tau\colon C \to A$. 
\qed
\end{prop}

As before, let $C$ be a differential graded coalgebra and $A$
a differential graded algebra.
Consider two twisting cochains $\tau_1,\tau_2 \colon C \to A$.
A homogeneous  morphism
$h \colon C \to A$ 
of the underlying graded modules
having degree zero
is a {\em homotopy of twisting cochains\/}
from $\tau_1$ to $\tau_2$, written
$h \colon \tau_1 \simeq \tau_2$, 
when
\begin{equation}
\tau_2\cup h = h \cup \tau_1  - D h.
\label{htw}
\end{equation}
In the augmented case,
a homotopy $h \colon \tau_1 \simeq \tau_2$
of twisting cochains is {\em augmented\/}
when $\varepsilon  h = \varepsilon$, in the coaugmnented case,
a homotopy $h$
of twisting cochains is {\em coaugmented\/}
when  $ h  \eta = \eta$ and, in the supplemented case,
a homotopy $h$
of twisting cochains is {\em supplemented\/}
when it is both
augmented and coaugmented.
An augmented  homotopy $h$  necessarily satisfies
the identity $\varepsilon  (Dh) = 0$ and
a coaugmented  homotopy $h$  satisfies
the identity $(Dh) \eta = 0$.

For $h\colon \tau_1 \simeq \tau_2\colon C \to A$,
the resulting identity
\begin{align*}
(Dh)\cap &= (h \cup \tau_1) \cap-(\tau_2\cup h) \cap
= h \cap \tau_1 \cap- \tau_2\cap h \cap
\end{align*}
says that the diagram
\begin{equation}
\begin{gathered}
\xymatrix{
C\otimes A \ar[d]_{d^\otimes+\tau_1\cap} \ar[r]^{h \cap}& C\otimes A 
\ar[d]^{d^\otimes +\tau_2 \cap}
\\
C\otimes A \ar[r]_{h \cap}& C\otimes A 
}
\end{gathered}
\end{equation}
is commutative. 
Hence $h \cap\colon C\otimes_{\tau_1} A \to
 C\otimes_{\tau_2} A$ is an $(\Id,\Id)$-morphism
of bundles.

\begin{rema}
{\rm
In \cite{zbMATH03921525},
the defining 
identity 
of a homotopy of twisting cochains
$k \colon \tau_1 \cong \tau_2$ reads 
\begin{equation}
\tau_1 \cup k =k \cup \tau_2 +Dk .
\end{equation}
}
\end{rema}

The following is again straightforward.
\begin{prop}
\label{gug3}
For a differential graded algebra $A$,
a  differential graded coalgebra $C$,
and two twisting cochains   $\tau_1, \tau_2\colon C \to A$,
the assignment to an $(\Id,\Id)$-bundle morphism  
\begin{equation}
\Psi \colon C \otimes_{\tau_1} A \to  C \otimes_{\tau_2} A
\end{equation} 
of the degree $0$ morphism 
\begin{equation}
h_\Psi\colon C 
\stackrel{\Id \otimes \eta} \longrightarrow C \otimes A
\stackrel \Psi\longrightarrow C \otimes A
\stackrel{\varepsilon \otimes \Id} \longrightarrow  A 
\end{equation}
of the underlying graded objects 
establishes, in the general,
 augmented, coaugmented, and supplemented case, a bijection between
homotopies of twisting cochains  from $\tau_1$ to $\tau_2$ and
 $(\Id,\Id)$-bundle morphisms
in such a way that
\begin{equation}
\Psi = h_\Psi \cap \colon  
C \otimes_{\tau_1} A
 \longrightarrow C \otimes_{\tau_2} A. 
\end{equation}
Thus
any  $(\Id,\Id)$-bundle morphism from
$C \otimes_{\tau_1} A$ to
$C \otimes_{\tau_2} A$
is of the kind 
\begin{equation*}
h\cap\colon C \otimes_{\tau_1} A
 \longrightarrow C \otimes_{\tau_2} A, 
\end{equation*}
for some uniquely determined 
homotopy $h \colon \tau_1 \simeq \tau_2$
of twisting cochains. \qed
\end{prop}

\begin{prop}
\label{gug2}
Let $A_1,A_2$ be differential graded algebras,
 $C_1,C_2$ differential graded coalgebras,
$\tau_1\colon C_1 \to A_1$ and $\tau_2\colon C_2 \to A_2$
twisting cochains,
  $\phi\colon A_1 \to A_2$ a morphism of augmented
differential graded algebras,
 $\chi \colon C_1 \to C_2$ a morphism of coaugmented
differential graded coalgebras,
and $h \colon \phi  \tau_1 \simeq \tau_2  \chi: C_1 \to A_2$.
Then the  composite
\begin{equation}
[\chi,h,\phi]\colon 
C_1\otimes_{\tau_1} A_1 \stackrel{\Id \otimes \phi}\longrightarrow 
C_1\otimes_{\phi  \tau_1} A_2\stackrel {h \cap}
\longrightarrow 
C_1\otimes_{ \tau_2  \chi} A_2
\stackrel{\chi \otimes \Id }\longrightarrow 
C_2\otimes_{\tau_2} A_2
\label{compo1}
\end{equation}
is a $(\phi,\chi)$-morphism  of bundles, and every  
 $(\phi,\chi)$-morphism 
$C_1\otimes_{\tau_1} A_1 \to C_2\otimes_{\tau_2} A_2$
of bundles arises in this manner from a suitable
homotopy $ \phi  \tau_1 \simeq \tau_2  \chi: C_1 \to A_2$
of twisting cochains.
This claim holds as well in the augmented, cooaugmented, and 
supplemented case.
\end{prop}

\begin{proof}
Any  $(\phi,\chi)$-morphism 
 \begin{equation*}
C_1\otimes_{\tau_1} A_1 \longrightarrow 
C_2\otimes_{\tau_2} A_2
\end{equation*}
factors as
 \begin{equation*}
C_1\otimes_{\tau_1} A_1 \stackrel{\Id \otimes \phi}\longrightarrow 
C_1\otimes_{\phi  \tau_1} A_2
\longrightarrow 
C_1\otimes_{ \tau_2  \chi} A_2
\stackrel{\chi \otimes \Id }\longrightarrow 
C_2\otimes_{\tau_2} A_2  . 
\end{equation*}
By Proposition \ref{gug3}, the middle bundle morphism
is of the kind
\begin{equation*}
C_1\otimes_{\phi  \tau_1} A_2\stackrel {h \cap}
\longrightarrow 
C_1\otimes_{ \tau_2  \chi} A_2,
\end{equation*}
for a unique homotopy 
$ \phi  \tau_1 \simeq \tau_2  \chi: C_1 \to A_2$
of twisting cochains.
\end{proof}

When a homotopy $h\colon \tau_1 \simeq \tau_2$,
viewed as a homogeneous 
degree zero member of the algebra $(\Hom(C,A), \cup)$, is invertible,
\eqref{htw} is equivalent to
\begin{equation}
\tau_2= h \cup \tau_1 \cup h^{-1} - (D h) \cup h^{-1}.
\label{htw2}
\end{equation}

As before, let $C$ be a differential graded coalgebra and $A$ a 
differential graded algebra.
Let $\TTT(C,A)$ denote the set of twisting cochains from $C$ to $A$.
The invertible members $\varphi$ of  $\Hom(C,A)$ 
(invertible degree zero morphisms of the underlying graded $R$-modules)
form a group $\GGG(C,A)$, and a straightforward verification shows that
the association
\begin{equation}
\begin{aligned}
\GGG(C,A) \times \Hom(C,A) &\longrightarrow \Hom(C, A),
\\
(\varphi,\rho) & \mapsto \varphi* \rho = \varphi \cup \rho \cup \varphi^{-1}
- (D\varphi) \varphi^{-1}
\end{aligned}
\end{equation}
yields an action of $\GGG(C,A)$ on $\TTT(C,A)$.

The assignment to $(C,A)$ of the $\GGG(C,A)$-orbits
$\DDD(C,A)= \TTT(C,A)/\GGG(C,A)$ 
is Berikashvili's functor under the present circumstances.
In the same vein, in the augmented, coaugmented, and supplemented  case,
let
 $\TTT_{\mathrm{aug}}(C,A)$, $\TTT_{\mathrm{coaug}}(C,A)$,
$\TTT_{\mathrm{supp}}(C,A)$,
denote the set of, respectively, 
augmented, coaugmented, supplemented
twisting cochains
and  let
\begin{align*}
\GGG_{\mathrm{aug}}(C,A)&=\{\varphi; \varepsilon \varphi = \varepsilon \}\subseteq \GGG(C,A),
 &\DDD_{\mathrm{aug}}(C,A)
&
=\TTT_{\mathrm{aug}}(C,A)/\GGG_{\mathrm{aug}}(C,A),
\\
\GGG_{\mathrm{coaug}}(C,A)&=\{\varphi;  \varphi \eta = \eta\}\subseteq \GGG(C,A),
&
\DDD_{\mathrm{coaug}}(C,A)
&
=\TTT_{\mathrm{coaug}}(C,A)/\GGG_{\mathrm{coaug}}(C,A),
\\
\GGG_{\mathrm{supp}}(C,A)&
=\{\varphi; \varepsilon \varphi = \varepsilon, \varphi \eta = \eta
\}\subseteq \GGG(C,A), 
&\DDD_{\mathrm{supp}}(C,A)
&
=\TTT_{\mathrm{supp}}(C,A)/\GGG_{\mathrm{supp}}(C,A).
\end{align*}

\begin{thm}
\label{param1}
For a  differential graded coalgebra 
 $C$ and a differential graded algebra $A$,
the assignment to a twisting cochain $t \colon C \to A$
of  the $A$-principal twisted tensor product
$C \otimes_t A$ on $C$
induces a bijection between $\DDD(C,A)$,
$\DDD_{\mathrm{aug}}(C,A)$,
$\DDD_{\mathrm{coaug}}(C,A)$,
$\DDD_{\mathrm{supp}}(C,A)$
and
isomorphism classes of, respectively,
general, augmented, coaugmented, and supplemented
$(C,A)$-bundles.
Thus the value $\DDD(C,A)$ of  Berikashvili's functor
$\DDD$ on $(C,A)$ 
parametrizes isomorphism classes of
$(C,A)$-bundles.
\end{thm}

\begin{proof}
Proposition \ref{gug3}
implies that  homotopic twisting cochains determine isomorphic
$(C,A)$-bundles, that is,
the map from 
$\DDD(C,A)$
to the set of isomorphism classes
of
$(C,A)$-bundles
is well defined and that, furthermore,
 this map 
is injective.
 Proposition \ref{gug1}
implies that this map 
is surjective.
\end{proof}

\begin{prop}
For   a cocomplete coaugmented differential graded coalgebra $(C,\eta)$ 
and a differential graded algebra $A$,
a degree zero morphism
$h \colon C \to A$ such that $h \eta = \eta \colon R \to A$
is invertible, i.e., belongs to $\GGG_{\mathrm{coaug}}(C,A)$.
\end{prop}

\begin{proof}
Write
$h = \eta  \varepsilon  + \widetilde h\colon C \to A$ such that  
$\widetilde h  \eta=0$.
Then
$h^{-1} =\eta  \varepsilon
+\sum_{j \geq 1}  (-  \widetilde h)^{\cup j}$.

Indeed,
for $p \geq 1$, the term $ \widetilde h^{\cup p}$ is zero
on $\FF_{p-1}C$ whence, restricted to $\FF_{p-1}C$, the infinite  sum
$\eta  \varepsilon 
+\sum_{j \geq 1}  (-  \widetilde h)^{\cup j}$
has only finitely many non-zero terms.
This implies the claim since $C$ is cocomplete, i.e., 
$C = \cup \FF_jC$. 
\end{proof}

Thus, for a cocomplete coaugmented coalgebra $C$ and a differential 
graded algebra $A$, there is no need to distinguish between
the members of $\GGG_{\mathrm{coaug}}(C,A)$ and coaugmented homotopies of
coaugmented twisting cochains from $C$ to $A$,
and this is, likewise, true in the supplemented case.

\section{Enriched categories, dg categories and dg cocategories}

Henceforth we  take every object,
e.~g., category, cocategory, etc. upon which we  carry out
an algebraic construction to be {\em small\/} without explicitly saying so.
Thus a (directed) graph (precategory, or quiver) 
of the kind mentioned above is supposed to be small.
However the universe, that is, the
closed monoidal category of 
graphs (or quivers  or precategories),
enriched in the closed monoidal category 
of chain complexes over $R$,
is not taken to be small.
For more details, the reader may consult, e.g., \cite{MR1567445},
\cite{MR651714},
\cite{MR294454}.

A small category is an {\em interpretation\/} 
\cite{MR1809685}
of the category axioms within set theory, cf. \cite{MR0354798}.
Thus an (oriented) {\em graph\/} $(O,A,s,t)$
consists of a set $O$ of {\em objects\/},
a set $A$ of {\em arrows\/}, and two
maps $s,t \colon A \to O$, the map $s$ being referred to
as {\em source\/} and the map $t$ as {\em target\/} map.
The {\em product\/} 
\[
A \times _OA = \{(g,f); g,f \in A,\, t(g)=s(f)\}
\]
of $A$ with itself {\em over $O$\/}
is the set of {\em composable\/} arrows.
When the set $O$ is fixed, it is common to refer to 
the graph $(O,A,s,t)$ as
an $O$-{\em graph\/}
\cite{MR0354798}.
A {\em morphism\/} of graphs and, likewise,
a {\em morphism\/} of $O$-graphs, is defined in the obvious way.
With this notion of morphism, graphs constitute a category
$\mathcal G$  and 
$O$-graphs form a subcategory $\mathcal G_O$ thereof.
A graph is {\em discrete\/} when its only arrows are the identity maps
between objects, so that the set of arrows coincides with its set of objects
and so that the source and target maps are necessarily
the identity.
A set $O$ determines a unique discrete graph 
$(O,O,\mathrm{Id},\mathrm{Id})$
in an obvious way and, with a slight abuse of notation, 
we denote this graph by $O$ as well
and  refer to it as the {\em discrete graph\/} $O$.

A (small) {\em category\/} is a graph 
$(O,A,s,t)$ together with  two  maps
\[
\mathrm{Id}\colon O \longrightarrow A,
\quad
c \colon A \times _O A \longrightarrow A,
\]
referred to as {\em identity\/} and {\em composition\/},
subject to the familiar constraints.
The set of arrows $A$ is then commonly referred to as that of 
{\em morphisms\/}. Thus, relative to the product over objects,
a category is a {\em monoid in the category of graphs\/}.
Likewise a {\em cocategory\/} is a comonoid in the category of graphs.
The discrete graph $O=(O,O,\mathrm{Id},\mathrm{Id})$
is a category in an obvious manner, the {\em discrete category\/}.
A {\em topological\/} category
is a small category having as  objects and morphisms topological spaces
with continuous  structure maps.
A {\em smooth\/} category
is, likewise,  a small category having as  objects and morphisms 
smooth manifolds
with smooth   structure maps.

Let  $R \mathcal G$ denote the category of
graphs enriched in the closed monoidal category $\modr$
of $R$-modules.
An $R$-{\em graph\/} is an object
of the category $R \mathcal G$.
 Given the two $R$-graphs $\mathcal A$ and $\mathcal B$,
a {\em morphism\/}  $f\colon \mathcal A \to \mathcal B$ of
$R$-{\em graphs\/} is defined in the obvious way: 
It consists of a map
\[
\mathrm{Ob}(f)\colon \mathrm{Ob}\mathcal A 
\longrightarrow \mathrm{Ob}\mathcal B
\]
and, for each ordered pair $(x,y)$ of objects of $\mathcal A$,
of a morphism
\[
f_{x,y}\colon \mathcal A(x,y) \longrightarrow  \mathcal B(fx,fy)
\]
of $R$-modules.
We refer to an $R$-graph having object set $O$ as
an $RO$-{\em graph\/}.  Plainly, $RO$-graphs constitute a subcategory
of $R \mathcal G$.

A particular $RO$-graph $R[O]$ arises from the 
discrete $O$-graph determined by $O$, with source and target maps
the identity map of $O$; we refer to this $R$-graph as the 
{\em discrete\/}  $RO$-graph.
With the obvious notions of composition and identity,
$R[O]$ becomes a category, indeed, 
$R[O]$ acquires a {\em ringoid\/} structure in an obvious manner
and, with the obvious interpretation of the term \lq\lq module\rq\rq,
a general $RO$-graph is then a {\em module\/} over $R[O]$;
frequently we will use the terminology 
$R[O]$-{\em module\/} rather than {\em module\/} over $R[O]$.
Thus $R[O]$-module and $RO$-graph are synonymous
notions.
We can realize the ringoid $R[O]$ as a functor from $O$ to
the category $\modr$ of $R$-modules.

Likewise, let $\crg$ be
the closed monoidal category 
of 
{\em graphs\/} enriched in the closed monoidal category 
$\chr$ of $R$-chain complexes.
With these preparations out of the way,
an $R$-{\em category\/} is
a unital associative algebra in the 
closed monoidal category
$R\mathcal G$ of $R$-graphs;
a {\em differential graded\/} $R$-{\em category\/}
(dg category)
is a unital associative algebra
in the category
$\crg$.
We will also refer to an object 
of $\crg$ having object set $O$ as an $R[O]$-{\em chain complex\/}.

The discrete $R$-graph $R[O]$ acquires obvious 
$R$-category, $R$-cocategory and, more generally,
dg category and dg cocategory structures.
The dg category $\mathcal A$ with object set $O$ being {\em unital\/}
signifies that the unit
$\eta\colon  R[O] \to \mathcal A$
is a morphism of dg categories.
This unit  encodes of course the identities in $\mathcal A$.
In the same vein, a (counital) {\em differential graded\/} $R$-cocategory
(dg cocategory) is a
counital coassociative coalgebra in the 
category $\crg$.
The dg cocategory $\mathcal C$ with object set $O$ being {\em counital\/}
signifies that the counit
$\varepsilon\colon \mathcal C\to R[O] $
is a morphism of dg cocategories.
A  dg category $\mathcal A$ with object set $O$ endowed with a morphism
$\varepsilon\colon \mathcal A\to R[O]$ of dg categories
such that $\varepsilon\eta$ is the identity of 
the dg category $R[O]$ is defined to be {\em augmented\/};
likewise
a  dg cocategory $\mathcal C$ with object set $O$ endowed with a morphism
$\eta\colon R[O] \to \mathcal C$ of dg cocategories
such that $\varepsilon\eta$ is the identity of 
the dg cocategory $R[O]$ is defined to be {\em coaugmented\/}.

\section{The nerve of a category}

\subsection{Preliminaries}
Let $\Ord$ denote the category of finite ordered sets
$[p]=(0,1,\ldots, p)$, $p \geq 0$, and monotone maps.
Let $\Cat$ be a category.
A {\em simplicial object\/}
in $\Cat$ is a contravariant functor from
$\Ord$ to $\Cat$;
a {\em cosimplicial object\/}
in $\Cat$ is a (covariant) functor from
$\Ord$ to $\Cat$.
The assignment to $[p]$ ($p \geq 0$) of the standard simplex 
$\nabla[p]$ yields a cosimplicial space $\nabla$.

Let $\Cat$ be a category, not necessarily small.
Let $\XX$ be an object of $\Cat$.
This object defines a \lq\lq trivially\rq\rq\ 
simplicial object in $\Cat$, and we denote this simplicial object
by $\XX$ again.
It has $X_p = X$, for $p \geq 0$, and every arrow the identity.

\subsection{The nerve}

The {\em nerve\/} $N\Cat $ of a small
category $\Cat$ is the simplicial set
having the objects as vertices, the morphisms as edges,
the triangular commutative diagrams  
as 2-simplicies, etc.
More formally, the nerve arises in the following way \cite{MR0232393}:
Regard an ordered set $S$  as a category  
with $S$ as set of objects 
and with just one morphism from $x\in S$ to $y\in S$ whenever
$x \leq y$.
Given the category $\mathcal C$, let
$N\mathcal C(S)$ be the set of functors from $S$ to $\mathcal C$.
As $S$ ranges over the finite ordered sets, this 
construction yields the {\em nerve\/}
of the category $\mathcal C$.

For convenience, here is the standard elementary description of the 
face and degeneracy operators on the degree $p$ constituent
$\NNN_p(\Cat) = \Mor_\Cat \times_O \ldots  \times_O \Mor_\Cat$ 
($p \geq 1$ factors) 
of the nerve $\NNN(\Cat)$ of $\Cat$,
with $\NNN_0(\Cat) = O$, and with the notation
$[x_0|\ldots|x_{p-1}] \in \NNN_p(\Cat)$ for $p \geq 1$ 
for the members of $\NNN_p(\Cat)$, 
in particular $p$ composable morphisms for $p \geq 2$:
\begin{equation}
\begin{aligned}
\pd_0[x]&= s(x) \in O,
\\
\pd_1[x]&= t(x) \in O,
\\
s_0(y)&= [\Id_y] \in \Mor_\Cat(y,y), 
\\
\pd_j[x_0|\ldots|x_{p-1}]&= 
\begin{cases}
[x_1|\ldots|x_{p-1}], & j = 0,
\\
[x_0|\ldots|x_{j-1} x_j| \ldots|x_{p-1}],
&1 \leq j \leq p-1, 
\\
[x_0|\ldots|x_{p-2}], &  j=p,
\end{cases}
\\
s_j[x_0|\ldots|x_{p-1}]&= [x_0|\ldots |x_{j-1}|\Id_{t(x_{j-1})}  |x_j| \ldots  |x_{p-1}],
\ 0 \leq j \leq p.
\end{aligned}
\end{equation}

\subsection{Bar and cobar constructions}
Let $\mathcal A$ be an augmented small dg category having object set $O$.
Its {\em nerve\/} $\NNN\mathcal A$ carried out
relative to the operation of taking the  tensor product over $R[O]$
(the appropriate coend)
is a simplicial
differential graded $R[O]$-module,
that is, a simplicial object in 
$\crg$; 
cf. \cite[IX.6 p.~226 ff.]{MR0354798}
 for the notion of coend.
Condensation, that is, totalization and normalization, yields the
differential graded $R[O]$-module
\begin{equation}
\mathcal B \mathcal A=|\NNN\mathcal A|,
\end{equation}
by construction, a differential graded
$R[O]$-graph having, in particular, $O$ as its set of objects.
The ordinary {\em Alexander-Whitney\/} diagonal $\Delta$ turns
$\mathcal B \mathcal A$ into a dg cocategory having $O$ as its set of 
objects. The resulting
dg cocategory $\mathcal B \mathcal A$ is the
{\em reduced normalized bar construction\/} for $\mathcal A$.
When $O$ consists of a single element so that 
$\mathcal A$ is an ordinary augmented dg algebra,
the cocategory $\mathcal B \mathcal A$ 
has a single object and
is the ordinary 
reduced normalized bar construction for $\mathcal A$.

An alternate construction of the 
reduced normalized bar construction relies on the observation that,
given the set $O$ and the dg category $\mathcal A$ with object set $O$, 
the functor which assigns to
a dg $R[O]$-module the induced $\mathcal A$-module
is left adjoint to the forgetful functor,
and the adjunction determines a comonad.
The bar construction then arises from
 the associated standard construction.

In the same vein, 
let $\mathcal C$ be a coaugmented small dg cocategory having object set $O$ .
The construction dual to the bar construction yields the
dg category
$\Omega \mathcal C$,
the {\em reduced normalized cobar construction\/} 
for $\mathcal C$.
When $O$ consists of a single element so that 
$\mathcal C$ is an ordinary coaugmented dg coalgebra,
$\Omega \mathcal C$ is the ordinary 
reduced normalized cobar construction for $\mathcal C$.
Similarly as before,
an alternate construction relies on the observation that
the appropriate adjunction determines a monad.
The cobar construction then arises from
the associated dual standard construction.

\subsection{The path object, its nerve, and twisted objects}
\label{pathob}
Let $\Cat$ be a small category.
Let $\XX$ be an object of $\Cat$.
Let $\PPP \XX$ be the category having
$\ob(\PPP \XX) =\XX$
and $\mor(\PPP \XX) =\XX\times \XX$, with
\begin{align*}
s,t \colon  \XX \times \XX 
&\longrightarrow \XX ,
\ s(x_1,x_2) = x_1, \  t(x_1,x_2) = x_2,\ x_1,x_2 \in \XX,
\\
\Id\colon   \XX  
&\longrightarrow\XX \times \XX,
\ 
\Id(x) = (x,x),\ x \in \XX,
\\
c \colon  \mor(\PPP \XX) \times_\XX \mor(\PPP \XX)
&\longrightarrow \mor(\PPP \XX),\ 
c((x_1,x_2),(x_2,x_3))= (x_1,x_3) ,\ x_1,x_2,x_3 \in \XX.
\end{align*}
Thus $\PPP \XX$ has
a unique morphism
between each pair of members of $\XX$.
To have a name, we refer to
$\PPP \XX$ as the {\em path category\/}
associated to $\XX$.

Let $PX = N\PPP \XX$, the nerve of $\PPP \XX$.
This is the familiar simplical object in $\Cat$ having
$(PX)_p =X^{\times (p+1)}$, for $p \geq 0$, with the standard
face and degeneracy maps
\begin{equation}
\begin{aligned}
\pd_j(y_0,\ldots,y_p)&= 
\begin{cases}
(y_1,\ldots,y_p), & j = 0,
\\
(y_0,\ldots, y_{j-1}, y_{j+1},\ldots,y_p),
&1 \leq j \leq p-1, 
\\
(y_0,\ldots,y_{p-1}), &  j=p,
\end{cases}
\\
s_j(y_0,\ldots,y_p)&=(y_0,\ldots, y_{j-1}, y_j, y_j, y_{j+1},\ldots,y_p), 
\ 0 \leq j \leq p,
\end{aligned}
\end{equation}
cf., e.g., \cite[Section 2 p.~294]{MR0345115}, 
also known as the {\em path object\/}
associated to $\XX$.

The path category
$ \PPP \Cat$ of $\Cat$ has 
\begin{align*}
\ob( \PPP \Cat)&=\Mor_\Cat
\\
\Mor( \PPP \Cat)&=\Mor_\Cat \times  \Mor_\Cat,
\ {\rm written\ as}\ \{(y_0,y_1),\  y_0, y_1 \in \Mor_\Cat \}
\\
s &\colon \Mor_\Cat \times  \Mor_\Cat \to \Mor_\Cat,\ 
s (y_0,y_1)= y_0,
\\
t &\colon \Mor_\Cat \times  \Mor_\Cat \to \Mor_\Cat,\ 
t (y_0,y_1)= y_1,
\\
\Id  &\colon \Mor_\Cat \to  \Mor_\Cat \times \Mor_\Cat ,\ 
\Id (y)= (y,y)
\\
c &\colon (\Mor_\Cat \times \Mor_\Cat) \times_{\Mor_\Cat}
(\Mor_\Cat \times \Mor_\Cat)
  \to \Mor_\Cat,\ 
c((y_0,y_1), (y_1,y_2) )= (y_0,y_2).
\end{align*}

Then $P \Cat = \NNN \PPP \Cat$ is the simplicial category
having $\ob_p(P \Cat) = O^{\times (p+1)}$
and  $\Mor_p(P \Cat) = \Mor_\Cat^{ (p+1)}$, for $p \geq 0$.

For $p \geq 1$, let $\NNN_p(\Cat) \times_O \Mor_\Cat$ denote the 
set of $(p+1)$-tuples
$[x_0|x_1|\ldots| x_{p-1}]x$ of composable morphisms in $\Cat$
 and interpret 
$\NNN_0(\Cat) \times_O \Mor_\Cat$ 
as $\Mor_\Cat$ 
in the obvious way.
The map
\begin{equation}
\begin{aligned}
\NNN_p(\Cat) \times_O \Mor_\Cat 
&= \Mor_\Cat^{\times_O p} \times_O \Mor_\Cat
\longrightarrow\Mor_\Cat^{ \times(p+1)}=\Mor_p(P \Cat)
\\
[x_0|x_1|\ldots| x_{p-1}]x &\mapsto (y_0,\ldots,y_p)
\\
(y_0,\ldots,y_p)&=(x_0 x_1\ldots x_{p-1} x,x_1\ldots x_{p-1}x, 
\ldots, x_{p-1}x, x)
\end{aligned}
\label{welldefined}
\end{equation}
is well defined.
Setting, for $p \geq 0$,
\begin{equation}
\begin{aligned}
\pd_j[x_0|\ldots|x_{p-1}]x&
=\begin{cases}
[x_1|\ldots|x_{p-1}]x, & j = 0,
\\
[x_0|\ldots|x_{j-1} x_j| \ldots|x_{p-1}]x,
&1 \leq j \leq p-1, 
\\
[x_0| \ldots|x_{p-2}]x_{p-1}x, &  j=p,
\end{cases}
\\
s_j[x_0|\ldots|x_{p-1}]x&
=
[x_0|\ldots |x_{j-1}|\Id_{t(x_{j-1})}  |x_j| \ldots  |x_{p-1}]x, 
\ 0 \leq j \leq p,
\end{aligned}
\end{equation}
defines a simplicial structure
on $\NNN(\Cat) \times_O \Mor_\Cat =(\NNN_p(\Cat) \times_O \Mor_\Cat )_{p\geq 0}$,
and
the canonical projection
\begin{equation}
\NNN(\Cat) \times_O \Mor_\Cat \longrightarrow \NNN(\Cat)
\label{canproj}
\end{equation}
and the maps
 \eqref{welldefined}
are compatible with the simplicial structures. 
We say that
$\NNN(\Cat)\times_O \Mor_\Cat$ is the 
{\em principal simplicial twisted object
associated to\/} $\Cat$.

The  degeneracy operators of $\NNN(\Cat)$ plainly determine
those  of $\NNN(\Cat)\times_O \Mor_\Cat$,
and this is also true of the face operators
apart from the last one.
For $p \geq 1$, the 
last face operator $\partial_p\colon \NNN_p(\Cat) \to 
\NNN_{p-1}(\Cat)$
together with the map 
\begin{equation}
\begin{aligned}
\rho_p &\colon  \NNN_p(\Cat) \longrightarrow \Mor_\Cat,
\\
\rho_p[x_0|\ldots|x_{p-1}] &= x_{p-1} ,
\end{aligned}
\end{equation}
determines the last face operator
of 
$\NNN(\Cat)\times_O \Mor_\Cat$
as
\begin{equation}
\begin{aligned}
\pd_p^\rho&\colon \NNN_p(\Cat) \times_O \Mor_\Cat \longrightarrow 
\NNN_{p-1}(\Cat) \times_O \Mor_\Cat
\\
\pd_p^\rho[x_0|\ldots|x_{p-1}]x
&=\pd_p[x_0|\ldots|x_{p-1}]\rho_p[x_0|\ldots|x_{p-1}] x .
\end{aligned}
\end{equation}
The
sequence
$\rho_1,\rho_2,\ldots$ of maps
$\rho_p \colon \NNN_p(\Cat) \to \Mor_\Cat$ ($p \geq 1$) enjoys the properties
\begin{align}
\rho(\partial_p b) \rho(b) &= \rho(\partial_{p-1}(b)), \ 
b \in \NNN_p(\Cat),\  p \geq 2,
\label{enjoy1}
\\
\rho(s_p(b))&= \Id_{t(\rho(b))} \in \Mor_\Cat(t(\rho(b),t(\rho(b)), 
\
b \in \NNN_p(\Cat),\  p \geq 1.
\label{enjoy2}
\end{align}
For a simplicial principal bundle, such a map has come to be known
as a {\em twisting function\/},
 cf. \eqref{enjoy3} and \eqref{enjoy4} in Section \ref{spb} below and 
 \cite{MR111028}, \cite{MR279808}, \cite[p.~406]{MR0301736}, 
\cite{MR111032}.
For $p \geq 2$, the commutative diagram
\begin{equation}
\begin{gathered}
\xymatrixcolsep{3pc}
\xymatrix{
\NNN_p(\Cat) \ar[r]^{\Delta\phantom{aaaaa}} 
\ar[d]^{\partial^{\mathrm{next\ to\ last}}}
& \NNN_p(\Cat) \times \NNN_p(\Cat) \ar[r]^{(\partial^\last,\rho_p)\phantom{aaa}}&
\NNN_{p-1}(\Cat) \times_O \Mor_\Cat  \ar[r]^{(\rho, \partial^\last)}
& \Mor_\Cat \times_O \Mor_\Cat \ar[d]^c
\\
\NNN_{p-1}(\Cat) \ar[rrr]_{\rho_{p-1}}& & & \Mor_\Cat 
}
\end{gathered}
\end{equation}
depicts property \eqref{enjoy1}.
We refer to 
$\rho$ as the {\em universal twisting function
for\/}  $\Cat$ and write the  
principal simplicial twisted object
$\NNN(\Cat)\times_O \Mor_\Cat$ 
associated to $\Cat$
 as  $\NNN(\Cat)\times_\rho \Mor_\Cat$.

\subsection{Group case}
\label{groupc}
Let $\GG$ be a group, discrete, topological, or a Lie group,
and consider $\GG$ as a category
with a single object having each morphism an isomorphism.
Below we simultaneously treat the discrete, topological and smooth cases
without further mention.
With $\GG$ substituted for $\Cat$ and for $\Mor_\Cat$,
consider the universal
 twisting function  $\rho \colon \NNN(\GG) \to \GG$ for $\GG$;
it  has constituents $\rho_p \colon \NNN_p(\GG) \to \GG$
($p \geq 1$),
and this twisting function 
determines the total space
 $\NNN(\GG) \times_\rho \GG$
of the
 simplicial  principal (right) $\GG$-bundle 
\begin{equation}
\NNN(\GG) \times_\rho \GG \to \NNN(\GG)
\label{sprr}
\end{equation}
(discrete, topological, smooth).

Consider
the path category
$\PPP \GG$ of $\GG$.
Segal  \cite{MR0232393} writes this category as $\overline \GG$.
The group $\GG$ acts freely 
on $\PPP \GG$
from the right, 
by right translation on $\ob (\PPP \GG)= \GG$ and 
diagonal  right translation on
$\mor(\PPP \GG) = \GG \times \GG$.
The association
\begin{equation}
\begin{aligned}
 (\GG \times \GG, \GG) &\longrightarrow (\GG, \{e\}),
\ 
(x_1,x_2,x) \mapsto (x_1x_2^{-1},e),\ x_1,x_2, x \in \GG,
\end{aligned}
\end{equation}
determines a smooth functor $\Pi \colon \PPP \GG \to \GG$ inducing
an isomorphism $(\PPP \GG)/\GG \to \GG$
of categories, a homeomorphism in the tolological case and a 
diffeomorphism in the smooth case.
Taking nerves, we obtain the  simplicial principal right $\GG$-bundle
\begin{equation}
N\Pi \colon N(\PPP \GG) \longrightarrow N\GG.
\label{simpprr}
\end{equation}
The (lean) geometric realization of $N\Pi$
yields the universal principal $\GG$-bundle
$E\GG \to B\GG$ over the classifying space $B\GG$
 \cite{MR0232393}, see also \cite{MR413122}.
The geometric realization of $N(\PPP \GG)$ is contractible
for completely formal reasons.

The constituent 
$N_0\Pi \colon N_0(\PPP \GG) \to N_0\GG$
is the trivial map $\GG \to \{e\}$, viewed as a principal
$\GG$-bundle and,
for $q \geq 1$, the constituent
$N_q\Pi \colon N_q(\PPP \GG) \to N_q\GG$
is the principal right $\GG$-bundle
\begin{equation}
N_q\Pi \colon \GG^{q+1} \longrightarrow \GG^q,\ 
 (y_0, y_1, \dots,y_q)
\mapsto
(y_0y_1^{-1}, y_1y_2^{-1},\dots,y_{q-1}y_q^{-1}).
\label{simpprqq}
\end{equation}
By construction, the group $\GG$ acts by diagonalwise right translation, 
that is
\begin{equation}
\begin{aligned}
 \GG^{q+1} \times \GG &\longrightarrow  \GG^{q+1},
\ 
 (y_0, y_1, \dots,y_q,y) \mapsto  (y_0y, y_1y, \dots,y_qy).
\end{aligned}
\end{equation}
The simplicial morphism  \eqref{welldefined} now takes the form
\begin{equation}
\NNN(\GG) \times _\rho \GG \longrightarrow \NNN(\PPP G),
\label{welldefined2}
\end{equation}
is plainly $\GG$-equivariant,
and has inverse
\begin{equation}
\begin{aligned}
\NNN(\PPP G) &\longrightarrow \NNN(\GG) \times _\rho \GG
\\
(y_0,\ldots,y_p)&\longmapsto [y_0y_1^{-1}| y_1y_2^{-1}|\ldots|y_{p-1}y_p^{-1}]y_p,\ 
p \geq 0.
\end{aligned}
\label{welldefined3}
\end{equation}
Thus \eqref{welldefined2} yields an isomorphism 
of simplicial principal right $\GG$-bundles
from
\eqref{sprr}  to \eqref{simpprr}.

\begin{rema}
{\rm
For the bar resolution,
it is common to refer to the $(p+1)$-tuples of the kind
$ (y_0, y_1, \dots,y_p)\in \GG^{p+1}$ 
as {\em homogeneous generators\/}
and to those of the kind
$[x_0|x_1|\ldots| x_{p-1}]x$
as {\em non-homogeneous generators\/}
\cite[IV.5 p.~ 119]{maclaboo}.
}
\end{rema}

\section{Simplicial principal bundles}
\label{spb}

Let $B$ be a simplicial set and $\KK$ a simplicial group.
A {\em twisting function\/} 
$\rho \colon B \to \KK$ consists of a sequence
$\rho_1,\rho_2,\ldots$ of maps
$\rho_p \colon B_p \to \KK_{p-1}$ ($p \geq 1$) subject to
\begin{align}
\rho(\partial_p b) \partial_{p-1}(\rho(b)) &= \rho(\partial_{p-1}(b)), 
\ b \in B_p, \ p \geq 2,
\label{enjoy3}
\\
\rho(s_p(b))&= e \in \KK_p, \ b \in B_p, \ p \geq 1,
\label{enjoy4}
\end{align}
 cf.  \cite{MR111028}, \cite{MR279808}, \cite[p.~406]{MR0301736}, 
\cite{MR111032}.
Similarly as before, the commutative diagram
\begin{equation}
\begin{gathered}
\xymatrixcolsep{3.5pc}
\xymatrix{
B \ar[r]^\Delta \ar[d]_{\partial^{\mathrm{next\ to\ last}}}& B \times B \ar[r]^{(\partial^\last,\rho)}&
B \times \KK  \ar[r]^{(\rho, \partial^\last)}
& \KK \times \KK \ar[d]^\mult
\\
B \ar[rrr]_{\rho}& & & \KK
}
\end{gathered}
\end{equation}
depicts property \eqref{enjoy3}.

\begin{rema}
{\rm
The terminology \lq twisting function\rq\ 
is consistent with the terminology
in Subsection \ref{groupc}
relative to an ordinary group 
via the assignment to an ordinary group 
of its associated trivially simplicial group.
}
\end{rema}

Let $\rho \colon B \to \KK$
be a twisting function. The 
{\em twisted cartesian product\/}
$B \times_\rho \KK$ is the simplicial $\KK$-set having
$B \times \KK$ as underlying graded object and
\begin{align*}
\pd_j&\colon B_p \times \KK_p \longrightarrow B_{p-1} \times \KK_{p-1},\ p \geq 1, 
\\
\pd_j(u,x)&= 
\begin{cases}
(\pd_j(u),\pd_j(x)), & 0 \leq j < p,
\\
(\pd_p(u),\rho_p(u)\pd_p(x)), & j=p,
\end{cases}
\\
s_j&\colon B_p \times \KK_p \longrightarrow B_{p+1} \times \KK_{p+1},\ p \geq 0, 
\\
s_j(u,x)&=(s_j(u),s_j(x)),\   0 \leq j \leq p,
\end{align*}
 \cite{MR111028} , \cite{MR279808}, \cite[p.~406]{MR0301736}, \cite{MR111032}.
The twisted cartesian product 
$B \times_\rho \KK$ is the total space of the resulting simplicial principal
$\KK$-bundle $\mathrm{pr}_B\colon B \times_\rho \KK\to B$.
Every simplicial principal (right) $\KK$-bundle
arises in this manner
 \cite{MR111028}, \cite{MR279808}, \cite{MR111032}; see also Remark \ref{W}
below.

\begin{prop}
\label{gug11}
For a simplicial set $B$ and
a simplicial group $\KK$,
the assignment to a simplicial principal bundle structure   
 on $B \times K$ of the degree $-1$ morphism 
\begin{equation}
\rho \colon B
\stackrel{\Id \times \{e\}} \longrightarrow B \times \KK
\stackrel {\pd^\last}\longrightarrow B \times \KK
\stackrel{\pr_\KK} \longrightarrow  \KK 
\end{equation}
of the underlying graded objects 
establishes a bijection between
twisting functions from $B$ to $\KK$ and
simplicial principal bundle structures
 $B \times \KK$ in such a way that
$B \times _\rho \KK$
recovers the simplicial principal bundle structure.
Thus
every simplicial principal bundle structure
is of the kind $B \times_\rho \KK$, for some uniquely determined twisting 
function
$\rho\colon B \to \KK$. 
\qed
\end{prop}

\begin{rema}
{\rm 
As in Subsection \ref{groupc} above,
 we here give preferred treatment  to the {\em last\/} face operator, 
as  in 
\cite{MR394720} and
\cite{MR111032, MR111033}.
This procedure is appropriate 
for principal bundles with  structure group acting
on the  total space
from the {\em right\/}
and simplifies comparison with the bar construction.

}
\end{rema}

The degree zero morphisms $\Mor_0(B,\KK)$ of the underlying graded
 sets from $B$ to $\KK$
form a group under
pointwise multiplication.
Let $\Mor_{-1}(B,\KK)$ 
denote the degree $-1$ morphisms
of the underlying graded
 sets from $B$ to $\KK$.
The association
\begin{equation}
\xymatrixcolsep{3.5pc}
\begin{aligned}
{}&\Mor(B,\KK) \times \Mor_{-1}(B,\KK) \longrightarrow \Mor_{-1}(B,\KK),
\
(\vartheta,\rho)  \mapsto \vartheta* \rho
\\
{}\vartheta* \rho&\colon \xymatrix{
B \ar[r]^{\Delta \phantom{aaa}} & 
B \times B  \times B\ar[r]^{(\pd^\last,\rho,\vartheta)\phantom{}}&
B \times \KK\times \KK \ar[r]^{(\vartheta,\Id,\pd^\last)} &
\KK \times \KK\times \KK\ar[r]^{\phantom{aaaaa}(\mathrm{left},\mathrm{right}^{-1})}&
\KK,
}
\end{aligned}
\label{act1}
\end{equation}
in formulas,
\begin{equation}
(\vartheta* \rho)(b)=\vartheta(\pd_p(b)) \rho(b)\vartheta(\pd_p(b))^{-1},
\ b \in B_p,
\end{equation}
yields an action of $\Mor(B,\KK)$ on $\Mor_{-1}(B,\KK)$.

\begin{lem}
\label{lem11}
Let $\rho_1\colon B \to \KK$ be a twisting function,
let $\vartheta \colon B \to \KK$
be a degree zero morphism of the underlying graded sets,
and let $\rho_2 = \vartheta * \rho_1\colon B \to \KK$.
Then $\rho_2$ is a twisting function if and only if
\begin{align}
\pd^\last\vartheta \pd^{\mathrm{ntlast}}&=\pd^\last \pd^\last\vartheta
\label{mort}
\end{align}
or, equivalently,
\begin{align}
\pd_{p-1}\vartheta(\pd_{p-1}(b))&=\pd_{p-1}(\pd_p(\vartheta(b))),\ b \in B_p,
\ p \geq 1.
\end{align}
\end{lem}

\begin{proof}
Let $b \in B$.
 Since $\rho_2 = \vartheta* \rho_1$, by  definition,
\begin{align*}
\rho_2(\pd_pb) &=
\vartheta(\pd_{p-1}\pd_pb) \rho_1 (\pd_pb)(\pd_{p-1}(\vartheta(\pd_pb)))^{-1}
\\
\partial_{p-1}(\rho_2(b)(\pd_p(\vartheta b)))&=
\partial_{p-1}(\vartheta(\pd_pb) \rho_1(b))
=[\pd_{p-1}(\vartheta(\pd_pb))] 
\pd_{p-1}( \rho_1(b))
\\
\rho_2(\pd_pb) \partial_{p-1}(\rho_2(b))\partial_{p-1}(\pd_p(\vartheta b))
&=
\vartheta(\pd_{p-1}\pd_pb) \rho_1 (\pd_pb)
\pd_{p-1}( \rho_1(b))
\\
&=
\vartheta(\pd_{p-1}
\pd_pb)\rho_1(\pd_{p-1}b) .
\end{align*}
On the other hand, still by definition,
\begin{align*}
\rho_2(\pd_{p-1}b)\pd_{p-1}\vartheta(\pd_{p-1}b)&
=\vartheta(\pd_{p-1}(\pd_pb)) 
\rho_1(\pd_{p-1}b).
\end{align*}
Hence 
\begin{equation*}
\rho_2(\pd_pb) \partial_{p-1}(\rho_2(b))=\rho_2(\pd_{p-1}b)
\end{equation*}
if and only if
$\partial_{p-1}(\pd_p(\vartheta b))= \pd_{p-1}\vartheta(\pd_{p-1}b)$.
\end{proof}

Let $\TTT(B,\KK)\subseteq \Mor_{-1}(B,\KK)$ denote the set of twisting 
functions from $B$ to $\KK$, and let 
$\Mort(B,\KK) \subseteq  \Mor_0(B,\KK)$ be
the subset consisting of those $\vartheta$ that are subject to \eqref{mort},
necessarily a subgroup.
The following is an immediate consequence of Lemma \ref{lem11}.

\begin{prop} The action {\rm {\eqref{act1}}}
restricts to
an  action of $\Mort(B,\KK)$ on $\TTT(B,\KK)$. \qed
\end{prop}

The assignment to a pair $(B,\KK)$ consisting of a simplicial set $B$ and  
a simplicial group $\KK$  of
the $\Mort(B,\KK)$-orbits 
\begin{equation}
\DDD(B,\KK)=\TTT(B,\KK)/\Mort(B,\KK)
\end{equation}
is a functor from the category of
pairs of the kind  $(B,\KK)$
to the category of sets.
This functor is the simplicial version of 
{\em Berikashvili\/}'s functor.

For a morphism $\phi\colon \KK_1 \to \KK_2$ of 
simplicial groups and a morphism $\chi \colon B_1 \to B_2$ 
of simplicial sets,
a $(\phi,\chi)$-{\em morphism\/} 
from the simplicial principal $\KK_1$-bundle
$E_1 \to B_1$
to the  simplicial principal $\KK_2$-bundle
$E_2 \to B_2$
is a commutative diagram
\begin{equation*}
\begin{gathered}
\xymatrix{
E_1\ar[r]^\Psi \ar[d]& 
E_2\ar[d]
\\
B_1\ar[r]_\chi & B_2
}
\end{gathered}
\end{equation*}
 in the category of simplicial sets
having $\Psi$ equivariant relative to $\KK_1$, 
the $\KK_1$-action on $E_2$
being via $\phi$.

Consider a  degree zero
morphism $\vartheta\colon B \to \KK$ of the underlying graded  sets. Define
the map
\begin{equation}
\vartheta *\colon
B\times \KK \stackrel{(\Delta,\Id)}\longrightarrow  B\times B \times\KK
\stackrel{(\Id,\vartheta,\Id)}\longrightarrow  B\times \KK \times\KK
\stackrel{(\Id,\mult)}\longrightarrow  B\times \KK,
\end{equation}
in formulas
\begin{equation}
\vartheta*(b,x) = (b, \vartheta(b) x).
\end{equation}
For twisting functions $\rho_1$ and $\rho_2$ from $B$ to $\KK$
such that $\vartheta* \rho_1 = \rho_2$, the map  $\vartheta *$ is an
$(\Id,\Id)$-morphism
\begin{equation}
\vartheta *\colon
B\times_{\rho_1} \KK \longrightarrow   B\times_{\rho_2} \KK 
\end{equation}
of simplicial principal bundles, necessarily an isomorphism.
Every $(\Id,\Id)$-morphism 
of simplicial principal $\KK$-bundles
from  $B \times_{\rho_1}\KK$ to $B \times_{\rho_2}\KK$
is necessarily of this kind,
cf. 
\cite[6.2 p. 100]{MR112135} (for left principal bundles).
The following is a consequence of Lemma \ref{lem11}.

\begin{prop}
\label{gug13}
For a simplicial set $B$,
a simplicial group $\KK$,
and two twisting functions   $\rho_1, \rho_2\colon B \to \KK$,
the assignment to an $(\Id,\Id)$-morphism  
$\Psi \colon B \times_{\rho_1} \KK\to  B \times_{\rho_2} \KK$ 
of simplicial principal bundles
of the  morphism 
\begin{equation}
\vartheta_\Psi\colon B 
\stackrel{\Id \times \{e\}} \longrightarrow B \times_{\rho_1} \KK
\stackrel \Psi\longrightarrow B \times_{\rho_2} \KK
\stackrel{\pr_\KK} \longrightarrow  \KK 
\end{equation}
of simplicial sets
establishes a bijection between
morphisms $\vartheta\colon B \to \KK$ of simplicial sets such that
$\vartheta* \rho_1 = \rho_2$
and $(\Id,\Id)$-morphisms of  simplicial principal bundles 
in such a way that
\begin{equation}
\Psi = \vartheta_\Psi * \colon   B \times_{\rho_1} \KK
\stackrel \Psi\longrightarrow B \times_{\rho_2} \KK. 
\end{equation}
Thus
any $(\Id,\Id)$-morphism of simplicial principal bundles from
$B \times_{\rho_1} \KK$ to
$B \times_{\rho_2} \KK$
is of the kind 
$\vartheta*\colon
B \times_{\rho_1} \KK \to
B \times_{\rho_2} \KK$, for some uniquely determined 
degree zero morphism $\vartheta \colon B \to \KK$ of 
the underlying graded
sets subject to
{\rm{\eqref{mort}}}. \qed
\end{prop}

\begin{prop}
\label{gug12}
Let $\KK_1,\KK_2$ be simplicial groups,
 $B_1,B_2$ simplicial sets,
$\rho_1\colon B_1 \to \KK_1$ and $\rho_2\colon B_2 \to \KK_2$
twisting functions,
  $\phi\colon \KK_1 \to \KK_2$ a morphism of simplicial groups,
 $\chi \colon B_1 \to B_2$ a morphism of simplicial sets,
and $\vartheta \colon  B_1 \to \KK_2$
a degree zero morphism of the underlying graded  sets such that
\begin{equation}
\vartheta* (\phi  \rho_1) = \rho_2  \chi \colon B_1 \to \KK_2.
\label{subj1}
\end{equation}
Then the  composite
\begin{equation}
[\chi,\vartheta,\phi] \colon
B_1\times_{\rho_1} \KK_1 \stackrel{\Id \times \phi}\longrightarrow 
B_1\times_{\phi  \rho_1} \KK_2\stackrel {\vartheta*}
\longrightarrow 
B_1\times_{ \rho_2  \chi} \KK_2
\stackrel{\chi \times \Id }\longrightarrow 
B_2\times_{\rho_2} \KK_2
\label{compo2}
\end{equation}
is a $(\phi,\chi)$-morphism  of simplicial principal bundles, and every  
 $(\phi,\chi)$-morphism of simplicial principal bundles from
$B_1\times_{\rho_1} \KK_1$ to $B_2\times_{\rho_2} \KK_2$
 arises in this manner from a suitable degree zero
morphism $ \vartheta \colon B_1 \to \KK_2$
of the underlying graded sets subject to {\rm \eqref{mort}} 
and {\rm \eqref{subj1}}.
\end{prop}

\begin{proof}
The argument is similar to that for the proof of Proposition \ref{gug2}.
We leave the details to the reader.
\end{proof}

\begin{thm}
\label{thm4}
For a simplicial set $B$ and a simplicial group $\KK$,
the assignment to a twisting function  $\rho \colon B \to \KK$
of  the $\KK$-principal twisted cartesian product
$B \times_\rho \KK$ 
induces a bijection between $\DDD(B,\KK)=\TTT(B,\KK)/\Mort(B,\KK)$
and isomorphism classes of
simplicial principal $\KK$-bundles on $B$.
Thus the value $\DDD(B,\KK)$ of  Berikashvili's functor $\DDD$ 
on $(B,\KK)$ 
parametrizes isomorphism classes of
simplicial principal $\KK$-bundles on $B$.

\end{thm}

\begin{proof}
Proposition \ref{gug13}
implies that  twisting functions
in the same $\Mort(B,\KK)$-orbit
 determine isomorphic simplicial principal
$(B,\KK)$-bundles, that is,
the map from 
$\DDD(B,\KK)$
to the set of isomorphism classes
of simplicial principal $\KK$-bundles on $B$s
is well defined and that, furthermore,
 this map 
is injective.
 Proposition \ref{gug11}
implies that this map 
is surjective.
\end{proof}

Recall that $|R(B\times \KK) |$, $|RB|$ and $|R\KK|$
denote the normalized chain complexes, 
each one, endowed with the Alexander-Whitney diagonal,
a differential graded coalgebra, and the latter furthermore,
 an augmented  differential graded algebra
under the multiplication map  which the group structure induces.
Consider the Eilenberg-Zilber contraction
\begin{equation}
 (\negthickspace\xymatrix{|RB|\otimes |R\KK|
\ar@<0.5ex>[r]^{\iota\phantom{a}}
&\ar@<0.5ex>[l]^{\ppi\phantom{a}}
{|R(B\times \KK)|} ,
h 
}\negthickspace) 
\label{ezcon1}
\end{equation}
\cite[Theorem 2.1 p.~51]{MR0065162},
written in \cite[Section 4 p.~404]{MR0301736}, with $\KK$ substituted 
for $F$,   as
\begin{equation}
 (\negthickspace\xymatrix{B\otimes \KK
\ar@<0.5ex>[r]^{\nabla\phantom{a}}
&\ar@<0.5ex>[l]^{f\phantom{a}}
{B\times \KK} ,
\Phi
}\negthickspace) .
\end{equation}
Let $\rho \colon B \to \KK$ be  twisting function. It
determines the perturbation
\begin{equation*}
\begin{aligned}
\pd^\rho\colon  |R(B\times \KK) | &\longrightarrow |R(B\times \KK) |,\\ 
\pd^\rho (b,x)&= (\pd_p(b),(\rho_p(b)-1) \pd_p(x)),\ (b,x) \in |R(B 
\times \KK)|_p,
\end{aligned}
\end{equation*}
of the differential $d$ on the normalized chain complex
$|R(B\times \KK) |$ of $B\times \KK$ so that
 $d_\rho = d + \pd^\rho$ is a  differential on $|R(B\times \KK) |$,
and we write the resulting chain complex as $|R(B\times_\rho \KK) |$.
The perturbation lemma 
yields a perturbation  $\dell^\rho$ 
of the tensor product differential $d^{\otimes}$
on 
$|RB|\otimes  |R\KK |$ together with 
 a contraction
\begin{equation}
(\negthickspace\xymatrix{(|RB|\otimes  |R\KK |,d^{\otimes}+ \dell^\rho)
\ar@<0.5ex>[r]^{\phantom{aaaa}\iota_\rho}
&\ar@<0.5ex>[l]^{\phantom{aaaa}g_\rho}
{|R(B\times_\rho \KK )|} ,
h_\rho 
}\negthickspace) ,
\label{contra1}
\end{equation}
and 
\cite[4.4 Lemma and  4.5 Lemma,   Section 4]{MR0301736}
assert that $(|RB|\otimes  |R\KK |,d^{\otimes}+ \dell^\rho)$ is a principal twisted object
or, equivalently, a $(|RB|, |R\KK |)$-bundle.
By Proposition \ref{gug1}, the composite
\begin{equation}
\tau^\rho\colon |RB| \stackrel{\Id \otimes \eta} 
\longrightarrow |RB| \otimes |RK| 
\stackrel{\dell^\rho} \longrightarrow  |RB| \otimes |RK|  \
\stackrel{\varepsilon \otimes \Id} \longrightarrow |RK|
\end{equation}
is a twisting cochain,
 the differential
 $d^{\otimes}+\dell^\rho$ is the twisted differential
 $d^{\tau^\rho}=d^{\otimes} +\tau^\rho \cap$,
see also \cite[p.~410]{MR0301736}, and the contraction \eqref{contra1}
takes the form
\begin{equation}
(\negthickspace\xymatrix{|RB|\otimes_{\tau^\rho}  |R\KK |
\ar@<0.5ex>[r]^{\iota_\rho}
&\ar@<0.5ex>[l]^{g_\rho}
{|R(B\times_\rho \KK )|} ,
h_\rho 
}\negthickspace) .
\label{contra2}
\end{equation}
By \eqref{1.1.1},
\begin{align*}
\dell^\rho &= \sum_{n\geq 1} \dell^\rho_n,
\
\dell^\rho_j=  \ppi\partial^\rho (h\partial^\rho)^{j-1}\iota 
=\ppi(\partial^\rho h)^j\partial^\rho\iota , \ j \geq 1.
\end{align*}
Thus, with the notation
\begin{equation*}
\tau^\rho_j\colon |RB| \stackrel{\Id \otimes \eta} 
\longrightarrow |RB| \otimes |RK| 
\stackrel{\dell^\rho_j} \longrightarrow  |RB| \otimes |RK|  
\stackrel{\varepsilon \otimes \Id} \longrightarrow |RK|,\ j \geq 1,
\end{equation*}
the twisting cochain $\tau^\rho$ takes the form
\begin{align}
\tau^\rho &= \sum_{j\geq 1} \tau^\rho_j.
\label{form1}
\end{align}

We now recall the \lq\lq Serre filtrations\rq\rq:
A member $(b,x)$ of $B \times \KK$ 
has filtration $\leq p$ when $b$ is the degeneration of 
 a member of $B_p$, and
\begin{equation*}
F_p(|RB| \otimes |R\KK|) = \sum_{0 \leq i \leq p} |RB|_i \otimes |R\KK|.
\end{equation*}
In terms of the Moore complexes, with the notation
$\pr \colon R(B \times_\rho K) \to RB$ for the projection,
\begin{align*}
F_p |RB| =F_p \Moo(RB) &=  \sum_{j \leq p} \Moo_j(RB)
\\
F_p \Moo(R(B \times_\rho K)) &= \pr^{-1} F_p \Moo(RB) . 
\end{align*}

\begin{thm}
\label{thm5}
For a simplicial set $B$ and a simplicial group $\KK$, 
the assignment to a twisting function $\rho \colon B \to \KK$
of the twisting cochain  $\tau^\rho \colon |RB|\to |R\KK|$ 
induces a map
\begin{equation}
\TTT(B,\KK) \longrightarrow \TTT_{\mathrm{aug}}(|RB|,|R\KK|)
\label{map1}
\end{equation}
that is natural in the data
in the following sense:
Let  $\rho_1\colon B \to \KK$ be a twisting function,
let $\vartheta \colon B \to \KK$ be a morphism of simplicial sets
subject to \eqref{mort},
and let $\rho_2 = \vartheta*\rho_1$, necessarily a twisting function.
Then $\vartheta$ induces  an augmented  homotopy
\begin{equation*}
h^\vartheta \colon \tau^{\rho_1}\simeq \tau^{\rho_2}
\colon |RB| \longrightarrow |RK|
\end{equation*}
 of twisting cochains
such that $ h^\vartheta \cap \colon 
|RB| \otimes_{\tau_1} |RK| \to|RB| \otimes_{\tau_2} |RK| $
is an $(\Id,\Id)$-isomorphism of augmented  $(|R B|, |R K|)$-bundles.
Hence the map {\rm \eqref{map1}}
passes to a  map
\begin{equation}
\DDD(B,\KK) \longrightarrow \DDD_{\mathrm{aug}}(|RB|,|R\KK|)
\end{equation}
and hence to a map from isomorphism classes of
simplicial principal $\KK$-bundles on $B$
to the isomorphism classes
of augmented  $(|R B|, |R K|)$-bundles.
\end{thm}

\begin{proof} The first claim is immediate.

Consider a twisting function $\rho_1\colon B \to \KK$,
let $\vartheta \colon B \to \KK$ be a morphism of simplicial sets
subject to \eqref{mort},
and let $\rho_2 = \vartheta*\rho_1$, necessarily a twisting function.
To simplify the notation, 
let $\tau_1 = \tau^{\rho_1}$, $\tau_2 = \tau^{\rho_2}$,
$\iota_1 = \iota_{\rho_1}$, $\iota_2 = \iota_{\rho_2}$,
$g_1 = g_{\rho_1}$, $g_2 = g_{\rho_2}$,
and let
\begin{equation*}
\Psi^\vartheta =g_2 |R(\vartheta *)| \iota_1\colon 
|RB|\otimes_{\tau_1}  |R\KK |
\longrightarrow
|RB|\otimes_{\tau_2}  |R\KK |,
\end{equation*}
necessarily a chain map.
Since $g_2 \iota_2 = |RB|\otimes_{\tau_2}  |R\KK |$, 
this map renders the diagram
\begin{equation*}
\begin{gathered}
\xymatrix{
|RB|\otimes_{\tau_1}  |R\KK | \ar[d]_{\Psi^\vartheta}
\ar[r]^{\iota_1}
&
{|R(B\times_{\rho_1} \KK )|} \ar[d]^{|R (\vartheta*)|}
\\
|RB|\otimes_{\tau_2}  |R\KK |
\ar[r]_{\iota_2}
&
{|R(B\times_{\rho_2} \KK )|}  
} 
\end{gathered}
\end{equation*}
commutative.

With $\KK \times \KK$ substituted for $\KK$, the 
Eilenberg-Zilber contraction \eqref{ezcon1}
reads
\begin{equation}
 (\negthickspace\xymatrix{|RB|\otimes |R\KK|\otimes |R\KK|
\ar@<0.5ex>[r]^{\iota_{\KK \times \KK}\phantom{}}
&\ar@<0.5ex>[l]^{\ppi_{\KK \times \KK}\phantom{}}
{|R(B\times \KK\times \KK)|} ,
h_{\KK \times \KK} 
}\negthickspace)
\label{ezcon2}
\end{equation}
and, applying the perturbation lemma yields,
for $k = 1,2$, the perturbed contraction
\begin{equation}
(\negthickspace\xymatrix{(|RB|\otimes_{\tau_k}  |R\KK |)\otimes |RK|
\ar@<0.5ex>[r]^{\widetilde \iota_k}
&\ar@<0.5ex>[l]^{\widetilde \ppi_k}
{|R((B\times_{\rho_k} \KK )\times K)|} ,
\widetilde h_k 
}\negthickspace) .
\label{contra12}
\end{equation}
In the same vein,
with $B \times B$ substituted for $B$, the 
Eilenberg-Zilber contraction \eqref{ezcon1}
reads
\begin{equation}
 (\negthickspace\xymatrix{|RB|\otimes |RB|\otimes |R\KK|
\ar@<0.5ex>[r]^{\iota_{B \times B}\phantom{}}
&\ar@<0.5ex>[l]^{\ppi_{B \times B}\phantom{}}
{|R(B\times B\times \KK)|} ,
h^{B \times B} 
}\negthickspace) 
\label{ezcon3}
\end{equation}
and, applying the perturbation lemma yields,
for $k = 1,2$, the perturbed contraction
\begin{equation}
(\negthickspace\xymatrix{|RB|\otimes(|RB|\otimes_{\tau_k}  |R\KK |)
\ar@<0.5ex>[r]^{\overline  \iota_k}
&\ar@<0.5ex>[l]^{\overline  \ppi_k}
{|R(B \times (B\times_{\rho_k} \KK)|} ,
\overline  h_k 
}\negthickspace) .
\label{contra13}
\end{equation}
Since $\vartheta*$ is an $(\Id,\Id)$-(iso)morphism of simplicial principal bundles,
the naturality of the constructions implies
that $\Psi^\vartheta$ is an $(\Id,\Id)$-(iso)morphism
of augmented $(|RB|,|RK|)$-bundles. Indeed, the diagram
\begin{equation*}
\begin{gathered}
\xymatrix{
 |RB|\otimes_{\tau_1}  |R\KK |)\otimes |RK|  
\ar[ddd]_{\Psi^\vartheta \otimes |RK|}
\ar[dr]_{\mult}
\ar[rrr]^{\widetilde \iota_1}
&&&
{|R((B\times_{\rho_1} \KK )\times \KK)|} 
\ar[ddd]^{|R ((\vartheta*)\times \KK)|}
\ar[dl]^{|R \mult|}
\\
&|RB|\otimes_{\tau_1}  |R\KK | \ar[d]_{\Psi^\vartheta}
\ar[r]^{\iota_1}
&
{|R(B\times_{\rho_1} \KK )|} \ar[d]^{|R (\vartheta*)|}
&
\\
&|RB|\otimes_{\tau_2}  |R\KK |
\ar[r]_{\iota_2}
&
{|R(B\times_{\rho_2} \KK )|} 
&
\\
|RB|\otimes_{\tau_2}  |R\KK |)\otimes |RK| \ar[ur]_{\mult}
\ar[rrr]_{\widetilde \iota_2}
&&&
{|R((B\times_{\rho_2} \KK )\times \KK)|} \ar[ul]^{|R \mult|}
} 
\end{gathered}
\end{equation*}
is commutative, and so is  the diagram
\begin{equation*}
\begin{gathered}
\xymatrix{
|RB|\otimes(|RB|\otimes_{\tau_1}  |R\KK |) 
\ar[ddd]_{|RB|\otimes\Psi^\vartheta }
\ar[rrr]^{\overline \iota_1}
&&&
{|R(B \times (B\times_{\rho_1} \KK ))|} 
\ar[ddd]^{|R (B \times (\vartheta*))|}
\\
&|RB|\otimes_{\tau_1}  |R\KK |\ar[ul]|-\Delta \ar[d]_{\Psi^\vartheta}
\ar[r]^{\iota_1}
&
{|R(B\times_{\rho_1} \KK )|}\ar[ur]|-{|R\Delta|} \ar[d]^{|R (\vartheta*)|}
&
\\
&|RB|\otimes_{\tau_2}  |R\KK |
\ar[r]_{\iota_2}\ar[dl]|-\Delta
&
{|R(B\times_{\rho_2} \KK )|} \ar[dr]|-\Delta
&
\\
|RB|\otimes(|RB|\otimes_{\tau_2}  |R\KK |)
\ar[rrr]_{\overline  \iota_2}
&&&
{|R(B \times (B\times_{\rho_2} \KK ))|} . 
} 
\end{gathered}
\end{equation*}
Since $\Psi^\vartheta$
is an $(\Id,\Id)$-morphism of augmented $(|RB|,|RK|)$-bundles,
by Proposition \ref{gug3},
we conclude that 
the composite
\begin{equation*}
h^\vartheta \colon|RB| \stackrel {|RB| \otimes \eta}
\longrightarrow |RB| \otimes |RK|
 \stackrel {\Psi^\vartheta}
\longrightarrow |RB| \otimes |RK|
\stackrel{\varepsilon \otimes |RK|}
\longrightarrow |RK|
\end{equation*}
is  an augmented  homotopy
\begin{equation*}
h^\vartheta \colon \tau^{\rho_1}\simeq \tau^{\rho_2}
\colon |RB| \longrightarrow |RK|
\end{equation*}
 of twisting cochains
such that $\Psi^\vartheta = h^\vartheta \cap \colon 
|RB| \otimes_{\tau_1} |RK| \to|RB| \otimes_{\tau_2} |RK| $.
\end{proof}

\begin{rema}
{\rm The  question emerges under which circumstances
the map from isomorphism classes of
simplicial principal $\KK$-bundles on $B$
to the isomorphism classes
of augmented  $(|R B|, |R K|)$-bundles
in Theorem \ref{thm5}
is an injection, surjection, or bijection.
}
\end{rema}

\begin{cor}
\label{cor6}
Let  $\chi \colon B_1 \to B_2$ be a morphism of simplicial sets,
 $\phi \colon \KK_1 \to \KK_2$ a morphism
of simplicial groups,
$\rho_1\colon B_1 \to \KK_1$ 
and $\rho_2\colon B_2 \to \KK_1$ 
twisting functions.
Then
  a degree zero morphism $\vartheta \colon B_1 \to \KK_2$
of the underlying graded sets subject to \eqref{mort} and \eqref{subj1} 
induces  an augmented  homotopy
\begin{equation*}
h^\vartheta \colon \phi \tau^{\rho_1}\simeq \tau^{\rho_2}\chi
\colon |RB_1| \longrightarrow |R\KK_2|
\end{equation*}
 of twisting cochains
such that, with the notation in {\rm \eqref{compo1}},
  {\rm \eqref{compo2}}, and {\rm\eqref{contra2}},
the diagram
\begin{equation*}
\begin{gathered}
\xymatrix{
|RB_1|\otimes_{\tau^{\rho_1}}  |R\KK_1| 
\ar[d]_{[|R\chi|,h^\vartheta,|R\phi|]}
\ar[r]^{\iota_{\rho_1}}
&
{|R(B_1\times_{\rho_1} \KK_1 )|} \ar[d]^{|R [\chi,\vartheta,\phi]|}
\\
|RB_2|\otimes_{\tau^{\rho_2}}  |R\KK_2 |
\ar[r]_{\iota_{\rho_2}}
&
{|R(B_2\times_{\rho_2} \KK_2 )|}  
} 
\end{gathered}
\end{equation*}
is commutative.  \qed
\end{cor}

Finally we consider simplicial principal bundles
having structure group 
an ordinary  group $\GG$, taken as a
trivially simplicial group.
In this case, for a simplicial set,
the defining properties 
\eqref{enjoy3} and \eqref{enjoy4}
of a twisting function
 $\rho \colon B \to \GG$
come down to \eqref{enjoy1} and \eqref{enjoy2},
with
$B$ substituted for $N\Cat$ and $\GG$ for $\Mor_\Cat$
and, $\GG$ being viewed as an ordinary group,  we use the notation 
$\TTT_{\mathrm{triv}}(B,\GG)$
for the set of maps
 $\rho \colon B \to \GG$
subject to  \eqref{enjoy1} and \eqref{enjoy2},
that is, for the set of twisting functions
from $B$ to $\GG$ when $\GG$ is taken as an ordinary group.
A straightforward verification establishes the following.
\begin{compl}
\label{triviallys}
For a simplicial set $B$ and a group $\GG$, taken as a
trivially simplicial group,
the restriction 
$\TTT(B,\GG) \to \TTT_{\mathrm{triv}}(B,\GG)$
is a bijection and, furthermore,
a member $\vartheta$ of $\Mort(B,\GG)$ satisfies the identity
\begin{align*}
\vartheta(b)&=\vartheta(\pd_0 \ldots \pd_{p-2}\pd_{p-1}(b)),\ b \in B_p,\ p \geq 1.
\end{align*}
Hence the restriction
\begin{equation}
\Mort(B,\GG) \longrightarrow \Map(B_0,\GG)
\end{equation}
is an isomorphism of groups and, in terms of the induced action
of $\Map(B_0,\GG)$ on $\TTT_{\mathrm{triv}}(B,\GG)$, 
\begin{equation}
\DDD(B,\GG)=\TTT_{\mathrm{triv}}(B,\GG)/\Map(B_0,\GG). \qed
 \end{equation} 
\end{compl}

\begin{rema}
\label{W}
{\rm
Let $K$ be a simplicial group. Recall that the $W$-construction
yields the universal  simplicial $K$-principal bundle
\cite[\S 17]{MR0056295}:
For $n \geq 0$,
\begin{equation}
(WK)_n = K_0 \times \dots \times K_n,
\end{equation}
with face and degeneracy operators given by the formulas
\begin{equation}
\begin{aligned}
\pd_0(x_0,\dots, x_n) 
&= (\pd_0x_1,\dots, \pd_0 x_n)\\
\pd_j(x_0,\dots, x_n) 
&=(x_0,\dots,x_{j-2}, x_{j-1}\pd_j x_j, \pd_j x_{j+1},\dots,\pd_j x_n),
\quad 1 \leq j \leq n
\\
s_j(x_0,\dots, x_n) 
&=(x_0,\dots,x_{j-1}, e,s_j x_j,s_j x_{j+1},\dots,s_j x_n),
\quad 0 \leq j \leq n;
\end{aligned}
\label{2.4}
\end{equation}
further,
$(\overline WK)_0 = \{e\}$ and, for $ n \geq 1$
\begin{equation}
(\overline WK)_n = K_0 \times \dots \times K_{n-1},
\end{equation}
with face and degeneracy operators given by the formulas
\begin{equation}
\begin{aligned}
\pd_0(x_0,\dots, x_{n-1}) 
&= (\pd_0x_1,\dots, \pd_0 x_{n-1}),
\\
\pd_j(x_0,\dots, x_{n-1}) 
&=(x_0,\dots,x_{j-2}, x_{j-1}\pd_j x_j, \pd_j x_{j+1},\dots,\pd_j x_{n-1}),
\\
&\qquad 1 \leq j \leq n-1,
\\
\pd_n(x_0,\dots, x_{n-1}) 
&=(x_0,\dots, x_{n-2}), 
\\
s_0(e) &= e 
\in K_0,
\\
s_j(x_0,\dots, x_{n-1}) 
&=(x_0,\dots,x_{j-1}, e,s_j x_j,s_j x_{j+1},\dots,s_j x_{n-1}),
\\
&\qquad 0 \leq j \leq n.
\end{aligned}
\label{2.5}
\end{equation}
The formulas \eqref{2.4} and \eqref{2.5} are consistent with those in 
\cite[A.14]{MR394720} for a simplicial algebra;
they differ from those in \cite{MR279808} (pp. 136 and 161)
where the constructions are carried out
with structure group acting from the {\em left\/}.
The maps
\begin{equation}
\rho_n = \pr_{K_{n-1}}\colon (\overline WK)_n =K_0 \times \dots \times K_{n-1}
\to K_{n-1},\ n \geq 1,
\end{equation}
assemble to a twisting function $\rho \colon \overline WK \to K$,
and the canonical map
$(\overline WK) \times_\rho K \to WK$
is an isomorphism of right $K$-simplicial sets.
Any simplicial principal $K$-bundle $P \to B$
admits a classifying map $B \to\overline WK$
and can hence be written as  
a twisted Cartesian product of the base $B$ with $K$.
For a modern account of the $W$-construction, 
with structure group acting from the left,
see
\cite[\S V.4 p.~269]{MR1711612}.

}
\end{rema}

\section{Principal bundles on the nerve of a category}
\label{pbs}

Let  $G$ be a group (discrete, topological, Lie, according 
to the case considered), viewed as a 
category with a single object, and let
$F\colon \Cat \to G$ be a functor.
Thus  $F$ assigns to every 
object of $\Cat$
 the identity element of $G$
and to every morphism $f \colon x_0\to x_1$ of
$\Cat$ a group element $F(f) \in G$
such that, whenever
$f_1\colon x_0\to x_1$ and $f_2\colon x_1\to x_2$,
\[
F(f_1f_2) = F(f_1)F(f_2).
\]
Since $F$ is a functor, this assignment extends to a morphism
\[
NF\colon N \Cat \longrightarrow NG
\]
of simplicial objects.
The pullback diagram
\begin{equation}
\begin{gathered}
\xymatrix{
\PPP_{\Cat,F} \ar[d]_{\Pi_F} \ar[r]^{F_\GG}& \PPP \GG \ar[d]^{\Pi}
\\
\Cat \ar[r]_{F} & \GG
}
\end{gathered}
\end{equation}
of categories characterizes the pullback category $\PPP_{\Cat,F}$.
By construction, the group $\GG$ acts freely from the right 
on $\PPP_{\Cat,F}$, and the canonical functor
$\PPP_{\Cat,F}/\GG \to \Cat$ is an isomorphism of categories.
Hence the nerve construction yields the simplicial principal right $\GG$-bundle
\begin{equation}
N\Pi_F \colon N(\PPP_{\Cat,F}) \longrightarrow N\Cat.
\label{nerveBG}
\end{equation}
The composite 
$\rho(N F) \colon N\Cat \to \GG$
of the induced morphism $N F \colon  N\Cat\to  N\GG$
of simplicial sets with the universal twisting function
$\rho \colon N(\GG) \to \GG$, cf. Subsection \ref{groupc},
yields the twisting function $N\Cat \to \GG$
which recovers the simplicial right $\GG$-set $N(\PPP_{\Cat,F})$,
the total object of \eqref{nerveBG}, as
the twisted cartesian product
$N\Cat \times_{\rho(N F)}\GG$ in such a way that
the bundle projection map 
$N \Pi_F\colon N\Cat \times_{\rho(N F)}\GG\to N\Cat$ 
amounts to  the canonical projection
to $N\Cat$.

The group $\Map(\ob_\Cat,\GG)$ of maps from $\ob_\Cat$ to $\GG$
with pointwise multiplication
acts on the set $\mathrm{Funct}(\Cat,\GG)$
of functors form $\Cat$ to $\GG$ via natural transformations of functors:
Consider two functors $F_1,F_2\colon \Cat \to \GG$
and a map $\Phi \colon \ob_\Cat \to \GG$
such that, for two objects $x$ and $y$ of $\Cat$ 
and a morphism $f$ from $x$ to $y$,
the diagram
\begin{equation}
\xymatrix{
F_1(x) \ar[r]^{F_1(f)} \ar[d]_{\Phi(x)}& F_1(y) \ar[d]^{\Phi(y)}
\\
F_2(x) \ar[r]_{F_2(f)} & F_2(y) 
}
\end{equation}
is commutative. Then $F_2 = \Phi(F_1)$ is the result
of the action by $\Phi \in \Map(\ob_\Cat,\GG)$ on $F_1$.

The assignment to a pair $(\Cat,\GG)$ consisting of a category $\Cat$
 and  
a group $\GG$  of
the $\Map(\ob_\Cat,\GG)$-orbits 
\begin{equation}
\DDD(\Cat,\GG)=\mathrm{Funct}(\Cat,\GG)/\Map(\ob_\Cat,\GG)
\end{equation}
is a functor from the category of
pairs of the kind  $(\Cat,\GG)$
to the category of sets.
This functor is the present version of 
{\em Berikashvili\/}'s functor.

\begin{thm} 
\label{bericat}
For a category
$\Cat$ and a group $\GG$,
the assignment to a functor $F \colon \Cat \to \GG$
of the simplicial principal bundle $N\Pi_F$,
cf. {\rm \eqref{nerveBG}}, induces  an injection  from $\DDD(\Cat,\GG)$
to the set of isomorphism classes of
simplicial principal $\GG$-bundles on 
$N\Cat$.
Thus the value $\DDD(\Cat,\GG)$ 
of Berikashvili's functor $\DDD$ on $(\Cat,\GG)$ 
parametrizes the isomorphism classes
of simplicial principal $\GG$-bundles on  $N\Cat$
of the kind $N\Pi_F$, for some functor $F \colon \Cat \to \GG$.
\end{thm}

\begin{proof}
Consider two functors $F_1,F_2\colon \Cat \to \GG$
and a map $\Phi \colon \ob_\Cat \to \GG$
such that $F_2 = \Phi(F_1)$.
By Complement \ref{triviallys},
the map $\Phi$ determines a unique map
$\vartheta \colon N\Cat \to \GG$
subject to \eqref{mort},
and 
\begin{equation} 
\vartheta*\colon N\Cat \times_{\rho(N F_1)}\GG \longrightarrow N\Cat \times_{\rho(N F_2)}\GG
\end{equation}
yields an isomorphism of simplicial principal $\GG$-bundles
on  $N\Cat$.
By Proposition \ref{gug13},
every isomorphism $N \Pi_{F_1} \to N \Pi_{F_2}$ 
of simplicial principal $\GG$-bundles
over the identity of  $N\Cat$ arises in this manner.
\end{proof}

\begin{examp}
\label{examp1}
{\rm
Let $B$ be a simplicial complex or, more generally, simplicial set,
and let $\mathcal C_B$ be the category 
whose set of objects $\ob_B$ is the set $S$ of simplices in $B$, 
with one morphism from $x\in S$ to $y\in S$ whenever
$x \leq y$, that is, whenever $x$ is a face of $y$.
As a simplicial set,
the nerve  $N\mathcal C_B$ of $\mathcal C_B$ is the barycentric subdivision
of $B$.
Thus a functor $F\colon \Cat_B \to \GG$ assigns to every simplex of $B$,
that is, to every vertex of the barycentric subdivision
$N \mathcal C_B$ of $B$, the identity element of $G$,
and to every oriented edge (1-simplex) $x_0\subseteq x_1$ of
$N \mathcal C_B$ a group element $F(x_0\subseteq x_1) \in G$
such that, whenever
$x_0\subseteq x_1$ and $x_1\subseteq x_2$,
\[
F(x_0\subseteq x_2) = F(x_0\subseteq x_1)F(x_1\subseteq x_2).
\]
Since $F$ is a functor, this assignment extends to a morphism
\[
NF\colon N \mathcal C_B \longrightarrow NG
\]
of simplicial objects.
Since, as a simplicial set,
the nerve  $N\mathcal C_B$ of $\mathcal C_B$ is the barycentric subdivision
of $B$,
by Theorem \ref{bericat},
applying Berikashvili's functor to $(\Cat_B,\GG)$
yields a set
$\DDD(\Cat_B,\GG)$ 
parametrizing the isomorphism classes
of simplicial principal $\GG$-bundles on  
the barycentric subdivision $N\mathcal C_B$
of $B$
of the kind $N\Pi_F$, for some functor $F \colon \Cat_B \to \GG$.

The lean geometric realization $||N\mathcal C_B||$ of the nerve 
$N\mathcal C_B$ of $\mathcal C_B$ is homeomorphic to the 
lean
geometric realization
$||B||$ of $B$. When $B$ is a simplicial complex
the homeomorphism is actually natural.
The proof where $B$ is a general simplicial set
is in \cite{MR222881} together with the observation  that
the homeomorphism cannot be taken to be natural
(\lq\lq Korollar\rq\rq\  p.~508).
Given the functor $F \colon \Cat_B \to \GG$,
geometric realization yields the principal $\GG$-bundle
\begin{equation}
||N\Pi_F|| \colon ||N(\PPP_{\Cat_B,F})|| \longrightarrow ||N\Cat_B|| .
\label{nerveBG2}
\end{equation}
Hence the value
$\DDD(\Cat_B,\GG)$ 
of Berikashvili's functor on $(\Cat_B,\GG)$
parametrizes  isomorphism classes
of  principal $\GG$-bundles 
of the kind \eqref{nerveBG2}
on the geometric realization of the barycentric subdivision of $B$.
}
\end{examp}

\begin{examp}
\label{examp2}
{\rm
Let  $M$ be a (topological, smooth, analytic, algebraic, according to the case considered)  manifold and
 $\UUU = \{ U_\lambda\}_{\lambda \in \Lambda}$
an open cover
of  $M$. Consider $\UUU$ as a partially orderet set,
with order relation by inclusion.
Let $\Cat_\UUU$ be the associated category.
It has $\Ob_{\Cat_\UUU} =\UUU$
and, for two members $U$ and $V$ of $\UUU$
with $U \subseteq V$
 a morphism from $U$ to $V$.

For a subset $\sigma$ of $\Lambda$, let 
$U_\sigma =\cap_{\alpha \in \sigma} U_\alpha$. 
The family $\{U_\sigma\}$ of the non-empty $U_\sigma$
as $\sigma$ ranges over finite subsets of $\Lambda$ 
forms an open cover of $M$ as well, and
we denote this open cover by $B\UUU$.
Here we use the letter $B$ as a mnemonic for \lq barycentric subdivision\rq, see below.
For $\tau \subseteq \sigma \subseteq \Lambda$, necessarily
$U_\sigma \subseteq U_\tau$.
The associated category $\Cat_{B\UUU}$ has
 $\Ob_{\Cat_{B\UUU}} =B\UUU$
and, for two objects $U_\sigma$ and $U_\tau$ of $B\UUU$
with $\tau \subseteq \sigma \subseteq \Lambda$,
 a morphism from $U_\sigma$ to $U_\tau$.

Following \cite[\S 4]{MR0232393}, we assign to 
$\Cat_{B\UUU}$ a 
(topological, smooth, analytic, algebraic, according to the case considered)
category $\MU$ as follows
(the notation in  \cite[\S 4]{MR0232393} is $X$ for $M$ and
$\mathbf X_{\mathbf U}$ for $\MU$):
Let $\Ob(\MU)=\coprod_{\sigma \in \Lambda}U_\sigma$,
 the disjoint union of the non-empty  $U_\sigma$,
for finite subsets $\sigma$ of $\Lambda$.
Thus the objects of $\MU$ are pairs
$(x,U_\sigma)$ with $x \in U_\sigma$, for finite subsets
$\sigma$ of $\Lambda$.
Define a morphism $(x,U_\sigma) \to (y,U_\tau)$
of $\MU$ to be an inclusion
$i\colon U_\sigma \to U_\tau$
with $i(x)=y$, for $\tau \subseteq \sigma$.
Hence,
for two finite subsets $\sigma_0 \subseteq \sigma_1$ of $\Lambda$, 
 a morphism $(x_1,U_{\sigma_1}) \to (x_0,U_{\sigma_0})$
of $\MU$ is the inclusion
$i\colon U_{\sigma_1} \to U_{\sigma_0}$
with $i(x_1)=x_0$.
For an ascending sequence
$\sigma_0 \subseteq \sigma_1 \subseteq \sigma_2$
and two morphisms
 $(x_2,U_{\sigma_2}) \to (x_1,U_{\sigma_1})$ 
and
$(x_1,U_{\sigma_1}) \to (x_0,U_{\sigma_0})$
of $\MU$, the composite of
 $(x_2,U_{\sigma_2}) \to (x_1,U_{\sigma_1})$ 
and
$(x_1,U_{\sigma_1}) \to (x_0,U_{\sigma_0})$
is the morphism
\begin{equation}
c((x_2,U_{\sigma_2}) \to (x_1,U_{\sigma_1}),(x_1,U_{\sigma_1}) \to (x_0,U_{\sigma_0})) =
 (x_2,U_{\sigma_2})  \to(x_1,U_{\sigma_1}) \to (x_0,U_{\sigma_0}).
\end{equation}
Thus the disjoint union
$\coprod_{[\sigma_0 \subseteq \sigma_1]}U_{\sigma_1}$
of the non-empty  $U_{\sigma_1}$
over the
inclusions
$\sigma_0 \subseteq \sigma_1\subseteq \Lambda$
of finite subsets of $\Lambda$ parametrizes the space
$\Mor(\MU)$ of $\MU$. Consider  an inclusion
$\sigma_0 \subseteq \sigma_1\subseteq \Lambda$ 
of finite subsets of $\Lambda$
and
let $U_{\sigma_0 \subseteq \sigma_1}$ denote the constituent of
$\Mor(\MU)$ which that inclusion parametrizes.
The following arrows characterize the so far missing pieces of
structure that turn $\MU$ into a category:
\begin{align*}
s\colon U_{\sigma_0 \subseteq \sigma_1} =U_{\sigma_1}
&\stackrel{\mathrm{incl}}\longrightarrow U_{\sigma_0},
\\
t\colon U_{\sigma_0 \subseteq \sigma_1} =U_{\sigma_1} 
&\stackrel{=}\longrightarrow U_{\sigma_1},
\\
\Id\colon U_{\sigma_0}
&\stackrel{=}\longrightarrow U_{\sigma_0 \subseteq \sigma_0} .
\end{align*}
The nerve $N(\MU)$ of $\MU$ is the simplicial manifold having, for $p \geq 0$,
\begin{equation}
N(\MU)_p=\coprod_{[\sigma_0 \subseteq \ldots \subseteq \sigma_p]}U_{\sigma_p}
\end{equation}
 the disjoint union
of the non-empty  $U_{\sigma_p}$, 
parametrized  by 
 ascending  sequences 
\begin{equation}
\sigma_0 \subseteq \ldots \subseteq \sigma_p \subseteq \Lambda
\label{ascend}
\end{equation}
of finite subsets of $\Lambda$.
By construction, for an ascending sequence of subsets of $\Lambda$
of the kind \eqref{ascend}, necessarily
\begin{equation}
U_{\sigma_p} = \cap_{0 \leq j \leq p} U_{\sigma_j} .
\end{equation}
The face and degeneracy operators are the corresponding inclusions.
See also [Section 2 p.~237]\cite{MR413122}.
As a simplicial manifold, the nerve $N\MU$ of $\MU$
is the barycentric subdivision of the ordinary nerve
$N\UUU$ of $\UUU$.

A functor $F \colon \MU \to \GG$ determines the corresponding
simplicial principal $\GG$-bundle \eqref{nerveBG}, viz.
\begin{equation}
N\Pi_F \colon N(\PPP_{\MU,F}) \longrightarrow N\MU ,
\label{nerveBG3}
\end{equation}
on 
$N\MU$. Thus, by Theorem \ref{bericat},
the set 
$\DDD(\MU,\GG)$ 
of Berikashvili's functor on $(\MU,\GG)$
parametrizes the isomorphism classes
of simplicial principal $\GG$-bundles on  
$N\MU$ of the kind $N\Pi_F$, for some functor $F\colon \MU \to \GG$.

View the manifold $M$ as a trivially simplicial manifold $M$.
Then the canonical projection $\MU \to M$ is a
morphism of simplicial manifolds.
A partition of unity of $M$ subordinate to $B \UUU$ induces a section 
$\iota \colon M \to ||N\MU||$,
and $\iota$ and  the projection $||N\MU|| \to M$ 
constitute a deformation retraction.
Thus, when $M$ is topological or smooth and paracompact,
the 
value of  Berikashvili's functor on $(\MU,\GG)$
parametrizes certain isomorphism classes
of  principal $\GG$-bundles on $M$;
if, furthermore, 
every  object of
$\MU$ is contractible,
every simplicial principal bundle on
$N \MU$ arises from a functor from $\MU$ to $\GG$, and
the value $\DDD(\MU,\GG)$ of  Berikashvili's functor on 
$(\MU,\GG)$
parametrizes  isomorphism classes
of  principal $\GG$-bundles on $M$.
Below we make this observation more precise.

For a functor $F \colon \MU \to \GG$,
by construction, applying  geometric realization
to \eqref{nerveBG3}
yields the ordinary principal $\GG$-bundle
\begin{equation}
||N\Pi_F|| \colon ||N(\PPP_{\MU,F})|| \longrightarrow || N\MU ||,
\label{nerveBG4}
\end{equation}
and
 the classifying map
\begin{equation}
|| NF || \colon || N \MU || \longrightarrow  || N \GG ||
\end{equation}
thereof arises as the geometric realization of
the morphism 
$ NF  \colon  N \MU \to  N \GG $
of simplicial objects
which the functor $F$ induces.
Thus, within the present framework,
the functor $F$ determines the associated bundle and classifying map
for free.
The composite with the section  $\iota \colon M \to ||N\MU||$
arising from a
partition of unity of $M$ 
subordinate to $B \UUU$
then yields the classifying map
for the resulting $\GG$-bundle on $M$ for free.
If
every  member of
$\UUU$ is contractible,
every principal $\GG$-bundle on $M$ arises in this way.

Consider the special case where 
the open cover $\UUU$ of
$M$ has a single member, i.e., $M$ itself.
The nerve $N \UUU$ of $\UUU$
is the trivially simplicial space
associated to $M$ (having, for $p \geq 0$,  as degree $p$ constituent
a copy of $M$ and every arrow the identity).
The category $\MU$ is the 
smooth
category having $M$ as its space of objects
and the nerve  $N \MU$ of $\MU$ is, likewise,
 the trivially simplicial space
associated to $M$.
(The barycentric subdivision oa a point is still a point.)
A functor $F \colon \MU \to \GG$
assigns to each object $x\in M$ of  $\MU$
the identity $e$ of $M$
and to the single morphism $\Id_M$ of $\MU$
the identity element of $\GG$.
Hence the resulting $\GG$-bundle on $M$ is trivial.

The observation 
in \cite[\S 4]{MR0232393}
that
$\GG$-transition functions relative to 
the open cover $B\UUU$ of
$M$
amount to a functor $F \colon \MU \to \GG$
 reconciles the present characterization  of 
a principal $\GG$-bundle in terms of an open cover 
 with the more classical one.

To put flesh on the bones of the last remark,
recall a {\em system of $\GG$-valued transition functions 
on $M$ relative to the open cover\/} $\UUU$ consists of a family of maps
$g_{\lambda,\mu}\colon U_\lambda \cap U_\mu \to \GG$ 
($\lambda,\mu \in \Lambda$) subject to 
(T1) below:
\begin{enumerate}
\item[{\rm (T1)}] For  $\lambda,\mu,\nu \in \Lambda$, the diagram
\begin{equation}
\xymatrixcolsep{3pc}
\xymatrix{
(U_\lambda \cap U_\mu \cap U_\nu) \times (U_\lambda \cap U_\mu \cap U_\nu)
\ar[r]^{\phantom{aaaaaaaaaaaaa}(g_{\lambda,\mu}, g_{\mu,\nu})}&\GG \times \GG \ar[r]^\mult & \GG
\\
U_\lambda \cap U_\mu \cap U_\nu
\ar[u]^\Delta \ar[urr]_{g_{\lambda,\nu}}& &
}
\end{equation}
\end{enumerate}
is commutative;
see, e.g., \cite[2.4 Definition Section 5.2 p.~63]{MR1249482}.

A system $\{g_{\lambda,\mu}\colon U_\lambda \cap U_\mu \to \GG\}$ 
of $\GG$-valued transition functions 
on $M$ relative to $\UUU$ 
necessarily satisfies
(T2)  and (T3) below \cite[Section 5.2 p.~63]{MR1249482}:
\begin{enumerate}
\item[{\rm (T2)}]
For  $\lambda\in \Lambda$,
the map $g_{\lambda,\lambda} \colon U_\lambda  \to \GG$ factors as
\begin{equation}
g_{\lambda,\lambda} \colon U_\lambda \to \{e\} \to \GG .
\end{equation} 
\item[{\rm (T3)}]
For  $\lambda, \mu\in \Lambda$,
the maps $g_{\lambda,\mu},g_{\mu,\lambda} \colon U_\lambda\cap U_\mu  \to \GG$
coincide.
\end{enumerate}

Two systems $\{g_{\lambda,\mu}\colon U_\lambda \cap U_\mu \to \GG\}$ 
and $\{g'_{\lambda,\mu}\colon U_\lambda \cap U_\mu \to \GG\}$ 
of $\GG$-valued transition functions 
on $M$ relative to $\UUU$ are {\em equivalent\/}
 \cite[2.6 Definition Section 5.2 p.~63]{MR1249482}
 provided there exist maps $r_\lambda\colon U_\lambda \to \GG$
($\lambda \in \Lambda$)
satisfying the relations (E) below:
\begin{enumerate}
\item[{\rm (E)}]
For $\lambda,\mu \in \Lambda$ and $y \in U_\lambda \cap U_\mu$,
\begin{equation}
g'_{\lambda,\mu}(y) = r_\lambda(y)^{-1}g_{\lambda,\mu}(y) r_\mu(y) .
\end{equation}
\end{enumerate}
This relation is an equivalence relation among
systems of $\GG$-transition functions on $\UUU$.
The equivalence classes form the first non-abelian cohomology
set $\mathrm{H}^1(\UUU, \GG)$
of $\UUU$ with coefficients in $\GG$; see, e.g.,
\cite[I.3 p.~40]{hirzeboo}.

Relative to the open cover $B \UUU$
of $M$, consider a system of  $\GG$-valued transition functions.
Such a system
takes the form
$\{g_{\sigma_0\subseteq\sigma_1}
\colon U_{\sigma_0\subseteq\sigma_1}\to\GG\}_{\sigma_0 \subseteq \sigma_1}$,
as $\sigma_0 \subseteq \sigma_1$ ranges over injections
of finite subsets of $\Lambda$.
The constraints (T1) -- (T3)
say that this system of transition functions
defines the functor
\begin{equation}
\begin{aligned}
F \colon \MU &\longrightarrow \GG
\\
F \colon \Ob(\MU) &\longrightarrow \{e\} \subseteq \GG
\\
F|_{U_{\sigma_0\subseteq\sigma_1}}=g_{\sigma_0\subseteq\sigma_1} \colon 
& U_{\sigma_0\subseteq\sigma_1}\longrightarrow \GG,
\end{aligned}
\label{functorF}
\end{equation}
and every such functor determines a 
system of  $\GG$-valued transition functions. Hence:

\begin{thm} 
\label{bericatG}
The assignment to a system of $\GG$-transition functions
on $B \UUU$ of the $\GG$-valued functor $F$
on $\MU$ which the assignment {\rm \eqref{functorF}}
characterizes
induces a bijection 
\begin{equation}
\mathrm{H}^1(B\UUU,\GG) \longrightarrow \DDD(\MU,\GG)
\end{equation}
onto the set  $\DDD(\MU,\GG)$ 
(Berikashvili's functor evaluated on $(\MU,\GG)$),
and this bijection  is natural in the data. \qed
\end{thm}

Suppose that $M$ is smooth and that $M$ admits a partition of unity
subordinate to the open cover $B\UUU$ of $M$.
When $M$ is paracompact, such a partition of unity is available.
A classical result says that
two principal $\GG$-bundles on $M$ arising from the same
open cover are isomorphic if and only if
the corresponding systems of $\GG$-transition functions are
equivalent, cf.
 \cite[2.7 Theorem Section 5.2 p.~63]{MR1249482}.
Thus, when  
each member of the open cover $\UUU$ of $M$ is contractible,
 the value  $\DDD(\MU,\GG)$ of Berikashvili's functor
on the pair $(\MU,\GG)$
parametrizes isomorphism classes of principal $\GG$-bundles on $M$.

}
\end{examp}

\begin{rema}
\label{nonabcoh}
{\rm
Theorem \ref{bericatG}
suggests one could view
the value $\DDD(C,A)$ of Berikasvili's functor
in Section \ref{bundles} on the pair $(C,A)$,
the value $\DDD(B,\KK)$ of that functor
in Section \ref{spb} on the pair $(B,\KK)$,
and the value $\DDD(\Cat,\GG)$ of that functor
 on the pair $(\Cat,\GG)$ in the present Section
as a kind of first non-abelian cohomology set.
This is consistent with the interpretation of 
a moduli space of flat connections
as  a first non-abelian cohomology space.

}
\end{rema}

\section{Twisting cochains and Berikashvili's functor in a categorical setting}
\label{tcs}

Let $\mathcal C$ be  a coaugmented
dg $R$-cocategory, with object set $O_{\mathcal C}$ and
with coaugmentation $\eta\colon R[O_{\mathcal C}] \to \mathcal C$,
and let $\mathcal A$ be an augmented dg $R$-category,
with object set $O_{\mathcal A}$ and  
augmentation $\varepsilon \colon \mathcal A \to R[O_{\mathcal A}]$.
A {\em twisting cochain\/}
is  a degree $-1$ morphism 
$t\colon \mathcal C \to \mathcal A$
of graded  $R$-graphs,
subject to the requirements
\begin{equation}
Dt = t \cup t,\ \varepsilon t = 0,\ t \eta = 0.
\label{tc1}
\end{equation} 
We will now unravel the meaning of this definition.
In particular, when $O_{\mathcal C}$ and $O_{\mathcal A}$
consist of a single element,
this notion of twisting cochain comes down to the standard one.

We return to the general case. The values of the diagonal $\Delta$
characterizing the cocategory structure  of $\mathcal C$ 
lie in the appropriate
{\em tensor square\/} of $\mathcal C$ with itself
in the category of differential graded $R[O_{\mathcal C}]$-modules.
This tensor square is the appropriate {\em coend\/} 
$\mathcal C\otimes_{R[O_{\mathcal C}]}\mathcal C$,
that is, $\mathcal C\otimes_{R[O_{\mathcal C}]}\mathcal C$ is the
differential graded $R[O_{\mathcal C}]$-module having the same set of objects $O_{\mathcal C}$ as
$\mathcal C$ and, given the  pair $(x,y)$ of objects,
the $R$-chain complex $(\mathcal C\otimes_{R[O_{\mathcal C}]}\mathcal C)(x,y)$
of arrows from $x$ to $y$ is the sum
\[
\bigoplus_{z \in O_{\mathcal C}}\mathcal C(x,z) \otimes \mathcal C(z,y), 
\]
the tensor product for each object $z$ 
being the ordinary tensor product of $R$-chain complexes.

A {\em homogeneous morphism\/}
$\alpha\colon \mathcal C \to \mathcal A$
of {\em graded \/} $R$-{\em graphs\/}
of a fixed degree $k\in \mathbb Z$ 
consists of (i) a set map
\[
\alpha\colon O_{\mathcal C}=\mathrm{Ob}(\mathcal C) 
\longrightarrow \mathrm{Ob}(\mathcal A)  =O_{\mathcal A}
\]
together with (ii), for each (ordered) pair $(x,y)$ of objects of $\mathcal C$,
a morphism
\[
\alpha_{x,y}\colon \mathcal C(x,y) \longrightarrow \mathcal A(\alpha(x),\alpha(y))
\]
of graded $R$-modules of degree $k$;
we will denote the degree of such a homogeneous morphism $\alpha$ by
$|\alpha|$. Given two homogeneous morphisms
$\alpha,\beta\colon \mathcal C \to \mathcal A$
of graded  $R$-graphs
which coincide on objects, that is,
\begin{equation}
\alpha=\beta\colon \mathrm{Ob}(\mathcal C) \longrightarrow \mathrm{Ob}(\mathcal A),
\label{ob1}
\end{equation}
their {\em cup product\/} $\alpha \cup\beta$
is the homogeneous morphism
$\alpha \cup\beta\colon \mathcal C \to \mathcal A$
of graded $R$-graphs
of degree $|\alpha|+|\beta|$ which, on objects, is given by
\eqref{ob1} and which,
for any (ordered) pair $(x,y)$ of objects of $\mathcal C$,
is given as the composite of the following three morphisms:
\begin{equation}
\begin{CD}
\mathcal C(x,y)
@>{\Delta}>>
\bigoplus_{z \in O_{\mathcal C}}\mathcal C(x,z) \otimes \mathcal C(z,y)
\\
@.
@.
\\
\bigoplus_{z \in O_{\mathcal C}}\mathcal C(x,z) \otimes \mathcal C(z,y)
@>{\alpha_{x,z}\otimes \alpha_{z,x}}>>
\bigoplus_{z \in O_{\mathcal C}} 
\mathcal A(\alpha(x),\alpha(z)) \otimes \mathcal A(\alpha(z),\alpha(y))
\\
@.
@.
\\
\bigoplus_{z \in O_{\mathcal C}}
 \mathcal A(\alpha(x),\alpha(z)) \otimes \mathcal A(\alpha(z),\alpha(y))
@>c>>  \mathcal A(\alpha(x),\alpha(y))
\end{CD}
\label{composition}
\end{equation}
This composition is well defined since,
given $f \in \mathcal C(x,y)$,
the value 
\[
\Delta(f) \in 
\bigoplus_{z \in O_{\mathcal C}}\mathcal C(x,z) \otimes \mathcal C(z,y)
\]
has at most finitely many non-zero components whence,
even though the lower-most arrow in \eqref{composition}
is, perhaps, not well defined when $O_{\mathcal C}$ is infinite,
the evaluation relative to the composition 
$c$ is well defined.

With these preparations out of the way,
a twisting cochain $t \colon \mathcal C \to \mathcal A$
consists of a set map
\[
t\colon \mathrm{Ob}(\mathcal C) \longrightarrow \mathrm{Ob}(\mathcal A)
\]
together with, for each (ordered) pair $(x,y)$ of objects of $\mathcal C$,
a morphism
\[
\alpha_{x,y}\colon \mathcal C(x,y) \longrightarrow \mathcal A(\alpha(x),\alpha(y))
\]
of graded $R$-modules of degree $-1$;
and $t$ is required to satisfy the identities \eqref{tc1}.

The following is an immediate generalization of the corresponding
facts for differential graded algebras and coalgebras:

\begin{prop}
Let $\mathcal C$ be  a coaugmented
dg $R$-cocategory
and $\mathcal A$ an augmented dg $R$-category.
A twisting cochain $t \colon \mathcal C \to \mathcal A$
induces a morphism
\begin{equation}
\overline t \colon \mathcal C \longrightarrow \mathcal B \mathcal A
\end{equation}
of coaugmented dg cocategories and a morphism
\begin{equation}
\overline t \colon \Omega \mathcal C \longrightarrow \mathcal A
\end{equation}
of augmented dg categories.
\end{prop}

In this proposition, the notation $\overline t$ is slightly abused.
Under these circumstances,
we will refer to either $\overline t$ as the {\em adjoint\/} of $t$.

Under the present circumstances,
Berikashvili's functor is still defined:
As before, let $\mathcal C$ be  a 
coaugmented dg $R$-cocategory, with object set $O_{\mathcal C}$ and with 
coaugmentation $\eta\colon R[O_{\mathcal C}] \to \mathcal C$, and let 
$\mathcal A$ be an augmented dg $R$-category, with object set 
$O_{\mathcal A}$ and  augmentation 
$\varepsilon \colon \mathcal A \to R[O_{\mathcal A}]$.
Let
\[
\varphi \colon O_{\mathcal C} \longrightarrow O_{\mathcal A}
\]
be a fixed map and let $\mathcal A_{\varphi}$ be the graded $R$-module
of homogeneous morphisms $\mathcal C \to \mathcal A$
of graded $R$-graphs which, on objects, are given by $\varphi$.
The standard Hom-differential $D$ turns
$\mathcal A_{\varphi}$ into an $R$-chain complex and the cup product
turns $\mathcal A_{\varphi}$ into a differential graded algebra, with
unit given by the composite
\begin{equation*}
\begin{CD}
\mathcal C
@>\varepsilon>>
R[O_\mathcal C]
@>{R \varphi}>>
R[O_\mathcal A]
@>{\eta}>>
\mathcal A .
\end{CD}
\end{equation*}
Denote the set of twisting cochains in $\mathcal A_{\varphi}$
by $\mathcal T(\mathcal A_{\varphi})$.
Next
let $G$ be the group of invertible elements of
$\mathcal A_{\varphi}^0$, and define the action 
\[
G \times \mathcal T(\mathcal A_{\varphi})
\longrightarrow \mathcal T(\mathcal A_{\varphi}), \ 
(x,y) \mapsto x * y,\ x \in G, y \in \mathcal T (\mathcal A_{\varphi}),
\]
of $G$
on $\mathcal T(\mathcal A_{\varphi})$,
by means of the formula
\begin{equation}
x *y = x y x^{-1} + (Dx) x^{-1}.
\label{action}
\end{equation}
This is well defined, that is,
given $x \in G$ and $y \in \mathcal T (\mathcal A_{\varphi})$, the value
$x *y$ satisfies the requirements \eqref{tc1} as well.
This is readily seen by a
straightforward calculation relying on the formulas
\[
(Dx) x^{-1}+x Dx^{-1} = 0, \quad
(Dx^{-1}) x +x^{-1} Dx = 0
\]
which, in turn, follow from $x x^{-1} = 1$.
We denote the set 
of orbits $(\mathcal T(\mathcal A_{\varphi})) /G$
by $\DDD(\mathcal A_{\varphi})$.
The assignment to $(\mathcal C,\mathcal A,\varphi)$ of
\begin{equation}
\DDD(\mathcal C,\mathcal A,\varphi)=\DDD(\mathcal A_{\varphi})
\end{equation}
is a functor from the category of
triples of the kind $(\mathcal C,\mathcal A,\varphi)$
to the category of sets.
This functor is the present version of 
{\em Berikashvili\/}'s functor.
The set $\DDD(\mathcal C,\mathcal A,\varphi )$,
i.e.,  Berikashvili's functor $\DDD$, evaluated at 
$(\mathcal C,\mathcal A,\varphi )$,
still ressembles a {\em moduli space of gauge 
equivalence classes of 
connections\/}.

All the previous examples for Berikashvili's functor
can be subsumed under this general version of  Berikashvili's functor.
For example, return to the circumstances of Example \ref{examp1}.
Let  $R\mathcal C_B$ be the corresponding category
enriched in the category of $R$-modules.
Then $\NNN (R\mathcal C_B)$
amounts to the induced simplicial $R[O]$-module
$RN\mathcal C_B$, and the dg cocategory
\begin{equation}
\mathcal BR\mathcal C_B =|\NNN R\mathcal C_B| = |RN\mathcal C_B|
\end{equation}
(where as before $|\,\cdot \,|$ refers to the normalized chain complex functor)
recovers the ordinary dg coalgebra
of normalized $R$-chains  on the simplicial set
$N\mathcal C_B$ (the barycentric subdivision of $B$)
in the following way:
The $R$-chain complex which underlies the dg coalgebra 
$|N\mathcal C_B|$ (arising from the normalized $R$-chains on 
$N\mathcal C_B$) is the direct sum
\[
|N\mathcal C_B| = \bigoplus_{x,y \in O} |RN\mathcal C_B|(x,y),
\]
and the Alexander-Whitney diagonal
\[
\Delta \colon |N\mathcal C_B| \longrightarrow |N\mathcal C_B| \otimes |N\mathcal C_B| 
\] 
is induced by the cocategory diagonal of 
$\mathcal BR\mathcal C_B$.
Given an ordinary augmented dg algebra $A$, 
viewed as an augmented  dg category with a single object,
a twisting cochain
\[
t \colon \mathcal BR\mathcal C_B \longrightarrow A
\]
in the sense of the present discussion
comes down to an ordinary twisting cochain
$|N\mathcal C_B| \to A$, defined on the ordinary
coaugmented dg $R$-coalgebra $|N\mathcal C_B|$
of normalized $R$-chains on the barycentric subdivision $N\mathcal C_B$  
of $B$.

The  general version of Berikashvili's functor deserves further study.
To this end, the appropriate framework is, perhaps, that of
$A_{\infty}$-categories due to Kontsevich and Fukaya:
An  $A_{\infty}$-category is an \lq\lq $A_{\infty}$-algebra 
with more than one object\rq\rq.
More precisely:
A unital $A_{\infty}$-{\em category\/} (over $R$) 
is a unital $A_{\infty}$-algebra in the 
closed monoidal category
$\crg$ 
(category of {\em graphs\/} enriched in the closed monoidal category 
$\chr$ of $R$-chain complexes) \cite{MR1270931}; thus,
a unital $A_{\infty}$-{\em category\/} (over $R$) 
with object set $O$ is
 a unital $A_{\infty}$-algebra
in the category
of 
$R[O]$-chain complexes.
Further,
given two augmented dg categories
$\mathcal A_1$ and $\mathcal A_2$, an $A_{\infty}$-{\em functor\/}
from $\mathcal A_1$ to $\mathcal A_2$
see, e.g., \cite[p.~688]{MR2028075} and the literature there, 
is a twisting cochain from $\mathcal A_1$ to 
$\mathcal A_2$ or, equivalently,
an ordinary (dg) functor from $\Omega\mathcal B A_1$ to $\mathcal A_2$.
Also, the  general version of Berikashvili's functor
is, perhaps, relevant for
the results and observations in
\cite{schlstas}
and
\cite{MR517083}.
\section*{Acknowledgements}

I am indebted to P. Goerss  for some email discussion
and to J. Stasheff for a number of comments.
This work was supported in part  by the Agence Nationale de la Recherche 
under grant ANR-11-LABX-0007-01 (Labex CEMPI).

\bibliographystyle{alpha}

\def\cprime{$'$} \def\cprime{$'$} \def\cprime{$'$} \def\cprime{$'$}
  \def\cprime{$'$} \def\cprime{$'$} \def\cprime{$'$} \def\cprime{$'$}
  \def\dbar{\leavevmode\hbox to 0pt{\hskip.2ex \accent"16\hss}d}
  \def\cprime{$'$} \def\cprime{$'$} \def\cprime{$'$} \def\cprime{$'$}
  \def\cprime{$'$} \def\Dbar{\leavevmode\lower.6ex\hbox to 0pt{\hskip-.23ex
  \accent"16\hss}D} \def\cftil#1{\ifmmode\setbox7\hbox{$\accent"5E#1$}\else
  \setbox7\hbox{\accent"5E#1}\penalty 10000\relax\fi\raise 1\ht7
  \hbox{\lower1.15ex\hbox to 1\wd7{\hss\accent"7E\hss}}\penalty 10000
  \hskip-1\wd7\penalty 10000\box7}
  \def\cfudot#1{\ifmmode\setbox7\hbox{$\accent"5E#1$}\else
  \setbox7\hbox{\accent"5E#1}\penalty 10000\relax\fi\raise 1\ht7
  \hbox{\raise.1ex\hbox to 1\wd7{\hss.\hss}}\penalty 10000 \hskip-1\wd7\penalty
  10000\box7} \def\polhk#1{\setbox0=\hbox{#1}{\ooalign{\hidewidth
  \lower1.5ex\hbox{`}\hidewidth\crcr\unhbox0}}}
  \def\polhk#1{\setbox0=\hbox{#1}{\ooalign{\hidewidth
  \lower1.5ex\hbox{`}\hidewidth\crcr\unhbox0}}}
  \def\polhk#1{\setbox0=\hbox{#1}{\ooalign{\hidewidth
  \lower1.5ex\hbox{`}\hidewidth\crcr\unhbox0}}}

\end{document}